\documentclass{amsart}
\usepackage{ verbatim, graphicx,amsmath,amssymb,amsthm} 
\usepackage{color}
\usepackage{xcolor}
\usepackage{ tikz} 

\usepackage{tikz, tikz-cd,}
\usepackage{hyperref}

\newtheorem{thm}{Theorem}
\newtheorem*{thm*}{Theorem}
\newtheorem{cor}[thm]{Corollary}
\newtheorem{lem}[thm]{Lemma}
\newtheorem{prop}[thm]{Proposition}

\newtheorem{rem}[thm]{Remark}

\DeclareMathOperator{\Aut}{Aut}
\DeclareMathOperator{\Out}{Out}

\renewcommand{\to}{\longrightarrow}

\newcommand{\Fq}{\mathbf{F}_q}

\newcommand{\Cl}{\mathrm{Cl}}
\newcommand{\GL}{\mathrm{GL}}
\newcommand{\PGL}{\mathrm{PGL}}
\newcommand{\PGammaL}{\mathrm{P\Gamma L}}
\newcommand{\SL}{\mathrm{SL}}
\newcommand{\PSL}{\mathrm{L}}

\newcommand{\GU}{\mathrm{GU}}

\newcommand{\SU}{\mathrm{SU}}
\newcommand{\PSU}{\mathrm{U}}
\newcommand{\Sp}{\mathrm{Sp}}
\newcommand{\PSp}{\mathrm{PSp}}

\newcommand{\POmega}{\mathrm{P}\Omega}
\newcommand{\F}{\mathbb{F}}

\newcommand{\Gammaone}{\Gamma}
\newcommand{\Sym}{\mathrm{S}}
\newcommand{\Alt}{\mathrm{A}}
\newcommand{\Mat}{\mathrm{M}}
\newcommand{\Sz}{\mathrm{Sz}} 
\newcommand{\cprod}{.}
\newcommand{\Cent}{\mathrm{Cent}}

\newcommand{\distinguish}[1]{#1'}

\begin{document}

\title{Perfect commuting graphs}%

\author{John R. Britnell}
\address{Department of Mathematics, Imperial College London, South Kensington Campus, London, SW7 2AZ, U.K.}
\email{j.britnell@imperial.ac.uk}

\author{Nick Gill}
\address{Department of Mathematics, University of South Wales, Treforest, CF37 1DL, U.K.}
\email{nicholas.gill@southwales.ac.uk}

\date{\today}%
\begin{abstract}
We classify the finite quasisimple groups whose commuting graph is perfect and we give a general structure theorem
for finite groups whose commuting graph is perfect.
\end{abstract}

\maketitle

\section{Introduction}

Let $\Gamma$ be a simple, undirected, finite graph with vertex set $V$. If $U\subseteq V$ then the \emph{induced
subgraph} of $\Gamma$ on $U$ is the graph $\Delta$ with vertex set $U$, and with two vertices connected in $\Delta$ if and only if they are connected in $\Gamma$.
The \emph{chromatic number} $\chi(\Gamma)$ is the smallest integer $k$ such that there exists a partition of
$V$ into $k$ parts, each with the property that it contains no two adjacent vertices. The \emph{clique number} $\Cl(\Gamma)$ is the size of
the largest complete subgraph of $\Gamma$. Clearly $\Cl(\Gamma)\le \chi(\Gamma)$ for any graph $\Gamma$. The graph
$\Gamma$ is {\it perfect} if $\Cl(\Delta)=\chi(\Delta)$ for every induced subgraph $\Delta$ of $\Gamma$.

Let $G$ be a finite group. The {\it commuting graph} $G$ is the graph $\Gammaone(G)$ whose vertices are the
elements of $G\backslash Z(G)$, with vertices joined by an edge whenever they commute. Some authors prefer not to
exclude the central elements of $G$, and nothing in this paper (barring the brief discussion of connectivity in \S\ref{s: literature}) depends significantly on which definition is used.

We are interested in classifying those finite groups $G$ for which the commuting graph $\Gammaone(G)$ is perfect. In this paper
we offer, in Theorem \ref{t: main}, a complete classification in the case of quasisimple groups. We use this
result to derive detailed structural information about a general finite group with this property,
in Theorem \ref{t: general}. The notation for quasisimple groups used in the statements of these theorems is explained
in \S\ref{s: background}.

\begin{thm}\label{t: main}
 Let $G$ be a finite quasisimple group and let $\Gammaone(G)$ be the commuting graph of $G$. Then $\Gammaone(G)$ is
 perfect if and only if $G$ is isomorphic to one of the groups in the following list:
 \begin{itemize}
   \item[] $\SL_2(q)$ with $q\geq 4$;
   \item[] $\PSL_3(2)$;
   \item[] $\PSL_3(4)$,\ $2\cprod \PSL_3(4)$,\ $3\cprod\PSL_3(4)$,\ $(2\times 2)\cprod\PSL_3(4)$,\
   $6\cprod \PSL_3(4)$,\ $(6\times 2)\cprod\PSL_3(4)$,\ $(4\times 4)\cprod\PSL_3(4)$,\ ${(12\times 4)\cprod\PSL_3(4)}$;
   \item[] $\Alt_6$, $3\cprod \Alt_6$, $6\cprod \Alt_6$;
   \item[] $6\cprod \Alt_7$;
   \item[] $\Sz(2^{2a+1})$ with $a\geq 1$;
   \item[] $2\cprod \Sz(8)$, $(2\times 2)\cprod\Sz(8)$.
 \end{itemize}
\end{thm}

Recall that a group is \emph{quasisimple} if it is a central extension of a simple group, and equal to its
derived subgroup. A \emph{component} of a finite group $G$ is a quasisimple subnormal subgroup of $G$. If $G$ is a group
such that $\Gammaone(G)$ is perfect, then every component of $G$ has a perfect commuting graph, and is therefore
isomorphic to one of the groups listed in Theorem \ref{t: main}. In fact we can say more.
\begin{thm}\label{t: general}
 Let $G$ be a finite group such that $\Gammaone(G)$ is perfect. Then $G$ has at most two components, and one of the
following statements holds.
\begin{enumerate}
\item $G$ has a single component, which is isomorphic to one of the groups in the following list:
\[
\SL_2(q),\ 6\cprod \Alt_6,\ (4\times 4)\cprod\PSL_3(4),\ (12\times 4)\cprod\PSL_3(4),\ \Sz(2^{2a+1}),\ 2\cprod \Sz(8),\ (2\times 2)\cprod\Sz(8).
\]
\item $G$ has a single component $N$, and the centralizer $\Cent_G(N)$ is abelian.
\item $G$ has two components $N_1,N_2$, each of which is isomorphic to one of the groups listed in (i), and such
that $\Cent_G(N_1N_2)$ is abelian.
\end{enumerate}
\end{thm}

We say that a group $G$ is an {\it AC-group} if the centralizer of every non-central
element of $G$ is abelian. An AC-group necessarily has a perfect commuting graph.
Information about these groups emerges naturally in the course of the proof of Theorem~\ref{t: main}, which
we summarize in the following result.

\begin{cor}\label{c: ac}
\begin{enumerate}
\item The finite quasisimple AC-groups are $6\cprod \Alt_6$ and $\SL_2(q)$ for $q\geq 4$.
\item If a finite AC-group $G$ has a component $N$, then $N$ is the unique component of $G$,
and the subgroup $NZ(G)$ has index at most $2$ in $G$.
\end{enumerate}
\end{cor}

An equivalent formulation of Corollary \ref{c: ac}(i) is that a finite quasisimple group $G$ is an
AC-group if and only if $G$ has a central subgroup $N$ such that $G/N\cong \SL_2(q)$ for some $q\geq 4$.
The case $G\cong 6\cprod \Alt_6$ conforms to this statement, since this group is isomorphic to $3\cprod \SL_2(9)$.

\subsection{Relation to the literature}\label{s: literature}

There has been a great deal of recent interest in  commuting graphs. One question which has attracted attention is the connectedness of $\Gammaone(G)$, where $G$ is a finite group. Recall that the \emph{prime graph} of a group $G$ is the graph whose vertices are the primes dividing $G$, with an edge between distinct vertices $p$ and $q$ if $G$ has a element of order $pq$. It has been shown by Morgan and Parker \cite{MP} that if $G$ is centreless, then $\Gammaone(G)$ is connected if and only if the prime graph of $G$ is connected. 

It is clear that if $G$ is quasisimple, and if $\Gammaone(G/Z(G))$ is not connected, then neither is $\Gammaone(G)$. The simple groups with connected prime graphs are known, from work of Williams \cite{Williams}, Kondrat'ev \cite{Kondratev}, and
Iiyori and Yamaki \cite{IY}. From their results, it can be seen that all of the groups listed in Theorem~\ref{t: main} have disconnected commuting graphs.

Another aspect of the commuting graph which has generated a lot of interest is the diameter of its connected components.
The construction by Hegarty and Zhelezov
\cite{HZ} of a $2$-group whose commuting graph has diameter $10$, has recently led Giudici and Parker \cite{GP} to a construction
of a family of $2$-groups with commuting graphs of unbounded diameter. This has answered (negatively) an influential
conjecture of
Iranmanesh and Jafarzadeh \cite{IJ}, that there was a universal upper bound for the diameter of a connected
commuting graph. Morgan and Parker \cite{MP} have shown, using the classification of finite simple groups,
that if $Z(G)$ is trivial then the diameter of any connected component of $\Gammaone(G)$ is at most~$10$.

Our interest in proving Theorems~\ref{t: main} and \ref{t: general} stems in part from earlier work with Azad \cite{abg} in
which we study a generalization of the commuting graph of a group $G$, namely the {\it $c$-nilpotency graph},
$\Gamma_c(G)$. This is the graph whose vertices are elements of $G$ with two vertices $g,h$ being connected if and
only if the group $\langle g, h\rangle$ is nilpotent of rank at most $c$. Ignoring central elements of $G$, the
commuting graph of $G$ is the same as the $1$-nilpotency graph.

In \cite{abg} we calculate the clique-cover number and independence number for the graphs $\Gamma_c(G)$ for various
simple groups $G$ and observe, in particular, that for these groups the two numbers coincide
(allowing the possibility that the graphs are perfect). This observation is part of the motivation for the current
paper and, moreover, suggests an obvious direction for further research: the question of which quasisimple groups
$G$ have perfect $c$-nilpotency graph, for $c>1$.

The question of which groups have perfect commuting graph has recently been asked in a blog entry
by Cameron \cite{cameronblog}. He gives an example of a finite $2$-group whose commuting graph is non-perfect.
A further motivation for this paper has been to provide at least a partial answer to his question.

The definition of an AC-group is a generalization of the better known notion of a {\it CA-group}, a group
for which the
centralizer of any non-identity element is abelian. Groups with this property arose naturally in early results
towards the classification of finite simple groups, and so they have significant historical importance.
It has been shown that every finite CA-group is a Frobenius group, an abelian group,
or $\SL_2(2^a)$ for some $a$ \cite{weisner}, \cite{bsw}, \cite{suzuki3}.

To a limited degree, Corollary~\ref{c: ac} extends this classical work on CA-groups. Of course our work,
unlike the work we have cited on CA-groups, depends on the classification of finite simple groups.
A proof of Corollary~\ref{c: ac} independent of the classification would be of considerable interest and significance.
There has been recent interest in AC-groups, and Corollary~\ref{c: ac} is related
to results from \cite{aam} and \cite{afo} in particular.

\subsection{Structure and methods}

The paper is structured as follows. In \S\ref{s: background} we state basic definitions and background results as well
as proving some straightforward general lemmas. In \S\ref{s: proof} we work through the different families of quasisimple
groups given by the classification, and we establish which finite quasisimple groups have perfect commuting graph, thereby
establishing Theorem~\ref{t: main}. In \S\ref{s: general} we prove Theorem~\ref{t: general} and Corollary~\ref{c: ac}. Finally, in \S\ref{s: extensions} we brefly discuss how Theorems~\ref{t: main} and \ref{t: general} may be extended, and we present some preliminary results in this direction.

For the most part the presentation of our arguments does not depend on computer calculation. We acknowledge,
however, that the computational
algebra packages GAP \cite{GAP} and Magma \cite{Magma} have been indispensable to us in arriving at our results, and
also that we have allowed ourselves to state many facts about the structure of particular groups, without proof or
reference, when such statements are easily verified computationally.

The methods we have chosen for groups of Lie type are by no means the only approach to the problem. We are grateful to an anonymous referee for 
pointing out an alternative way of ruling out groups of large Lie rank. It is clear that if $G$ is a group of Lie rank at least $5$, then its Dynkin diagram contains an independent set of vertices of size $3$. This implies the existence of three subgroups of $G$ isomorphic to $\SL_2(q)$, which centralize one another. But Proposition~\ref{p: at most 2 components} now implies that $\Gammaone(G)$ is not 
perfect.  

By `group' we shall always mean `finite group'.
By `simple group' we shall always mean `non-abelian finite simple group'.

\subsection{Acknowledgments}

The second author was a visitor at the University of Bristol while this work was undertaken and would like to thank the members of the Department of Mathematics for their hospitality. Both authors would like to thank Professor J\"urgen M\"uller for help with the GAP calculations discussed in \S\ref{s: extensions}.

\section{Background}\label{s: background}

In this section we gather together relevant background material, as well as proving some basic lemmas.

\subsection{Graphs}

In this paper all graphs are finite, simple and undirected. Let $\Gamma=(V,E)$ be such a graph (with vertex set $V$
and edge set $E$).
\begin{enumerate}
 \item The {\it chromatic number of} $\Gamma$, $\chi(\Gamma),$ is the smallest number of colours required to colour
 every vertex of $\Gamma$ so that neighbouring vertices have different colours.
 \item The {\it order} of $\Gamma$ is $|V|$.
 \item The {\it clique number of} $\Gamma$,  $\Cl(\Gamma)$, is the order of the largest complete subgraph of $\Gamma$.
 \item An {\it induced subgraph of} $\Gamma$ is a graph $\Lambda=(V', E')$ such that $V'\subseteq V$ and there is an
 edge between two vertices $v$ and $v'$ in $\Lambda$ if and only if there is an edge between $v$ and $v'$ in $\Gamma$.
 \item We say that $\Gamma$ is {\it perfect} if $\chi(\Lambda)=\Cl(\Lambda)$ for every induced subgraph of $\Lambda$.
\end{enumerate}
We have already noted in the introduction that every graph $\Gamma$ satisfies $
\chi(\Gamma) \geq \Cl(\Gamma)$.

To state the two most important theorems concerning perfect graphs, we require some terminology. The
{\it complement} of $\Gamma$ is the graph $\Gamma^c$ with vertex set $V(\Gamma)$, in which an edge connects
two vertices if
and only if they are not connected by an edge in $\Gamma$. A \emph{cycle} is a finite connected graph
$\Gamma$ such that every vertex has valency $2$. A \emph{$k$-cycle} is a cycle of order $k$.

\begin{thm*}[Weak Perfect Graph Theorem \cite{lov}]
 A graph $\Gamma$ is perfect if and only if the complement of $\Gamma$ is perfect.
\end{thm*}

\begin{thm*}[Strong Perfect Graph Theorem \cite{spgt}]
A graph $\Gamma$ is perfect if and only if it has no induced subgraph isomorphic either to a cycle of odd order
at least $5$, or to the complement of such a cycle.
\end{thm*}

We shall say that a subgraph $\Delta$ of $\Gamma$ is \emph{forbidden} if $\Delta$ is an induced subgraph of $\Gamma$, and
either $\Delta$ or $\Delta^c$ is isomorphic to a cycle of odd order at least $5$. Figure~\ref{f: three
smallest} shows the three smallest forbidden subgraphs.
(We note that the complement of a $5$-cycle is another $5$-cycle.)

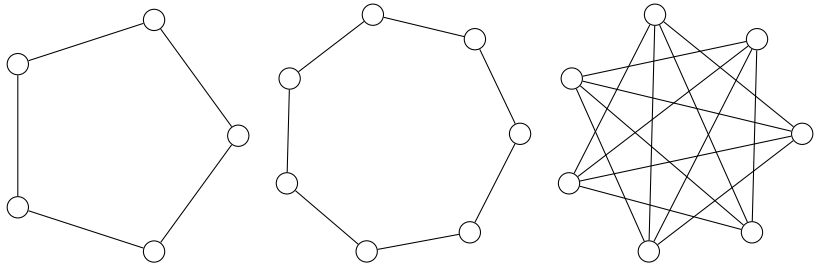
\begin{figure}[ht]
\centering
\begin{tabular}{ccc}

\newcount\mycount

\begin{tikzpicture}[scale=0.3]
  \foreach \number in {1,...,5}{
        \mycount=\number
        \advance\mycount by -1
  \multiply\mycount by 72
        \advance\mycount by 0
      \node[draw,circle,inner sep=0.10cm] (N-\number) at (\the\mycount:5.4cm) {};
    }
  \path (N-1) edge (N-2);
  \path (N-2) edge (N-3);
  \path (N-3) edge (N-4);
  \path (N-4) edge (N-5);
  \path (N-5) edge (N-1);

    \end{tikzpicture}

    &
  \newcount\mycount

\begin{tikzpicture}[scale=0.3]

  \foreach \number in {1,...,7}{
        \mycount=\number
        \advance\mycount by -1
  \multiply\mycount by 51
        \advance\mycount by 0
      \node[draw,circle,inner sep=0.10cm] (N-\number) at (\the\mycount:5.4cm) {};
    }
  \path (N-1) edge (N-2);
  \path (N-2) edge (N-3);
  \path (N-3) edge (N-4);
  \path (N-4) edge (N-5);
  \path (N-5) edge (N-6);
  \path (N-6) edge (N-7);
  \path (N-7) edge (N-1);
  \end{tikzpicture}
&

  \begin{tikzpicture}[scale=0.3]

  \foreach \number in {1,...,7}{
        \mycount=\number
        \advance\mycount by -1
  \multiply\mycount by 51
        \advance\mycount by 0
      \node[draw,circle,inner sep=0.10cm] (N-\number) at (\the\mycount:5.4cm) {};
    }
  \path (N-1) edge (N-3);
  \path (N-2) edge (N-4);
  \path (N-3) edge (N-5);
  \path (N-4) edge (N-6);
  \path (N-5) edge (N-7);
  \path (N-6) edge (N-1);
  \path (N-7) edge (N-2);
  \path (N-1) edge (N-4);
  \path (N-2) edge (N-5);
  \path (N-3) edge (N-6);
  \path (N-4) edge (N-7);
  \path (N-5) edge (N-1);
  \path (N-6) edge (N-2);
  \path (N-7) edge (N-3);

  \end{tikzpicture}
\end{tabular}
 \caption{The three forbidden subgraphs of smallest order.}
 \label{f: three smallest}
\end{figure}

The term \emph{Berge graph} has also been used to mean a graph with no forbidden subgraphs (and an alternative
statement of the Strong Perfect Graph Theorem is that the class of Berge graphs and the class
of perfect graphs are the same). In fact the arguments presented in \S\ref{s: proof} directly
characterize those quasisimple groups $G$ for which $\Gamma(G)$ is a Berge graph. It is the Strong Perfect Graph
Theorem which allows us to express these results in terms of perfect graphs.

\subsection{Commuting graphs}

Let $G$ be a finite group. We defined the commuting graph $\Gammaone(G)$ in the introduction. It is worth
justifying here the assertion that the presence or absence as vertices of the central elements of $G$ has no affect on
whether $\Gammaone(G)$ is perfect. Let $\Gammaone'(G)$ be the graph with vertex set $G$ and an edge $\{g,h\}$ if and
only if $gh=hg$. We note that $\Gammaone(G)$ is the induced subgraph of $\Gammaone'(G)$ on the vertices $G\setminus Z(G)$.

\begin{lem}\label{l: no diff}
 $\Gammaone(G)$ is perfect if and only if $\Gammaone'(G)$ is perfect.
\end{lem}
\begin{proof}
Suppose that $\Gammaone'(G)$ is perfect and $\Lambda$ is an induced subgraph of $\Gammaone(G)$. Then $\Lambda$ is an
induced subgraph of $\Gammaone'(G)$ and so $\chi(\Lambda)=\Cl(\Lambda)$ and $\Gammaone(G)$.

Conversely suppose that $\Gammaone(G)$ is perfect and $\Lambda'$ is an induced subgraph of $\Gamma'(G)$. Then
$V(\Lambda') = V(\Lambda)\cup V_Z$ where $\Lambda$ is an induced subgraph of $\Gamma(G)$ and $V_Z$ is a set of central
elements. Now
\[
\chi(\Lambda')=\chi(\Lambda)+|V_Z| = \Cl(\Lambda)+|V_Z| = \Cl(\Lambda')
\]
and we conclude that $\Gammaone'(G)$ is perfect.
\end{proof}


It will be convenient to extend our notation in the following way: if $\Omega\subseteq G$, then we write
$\Gammaone(\Omega)$ for the induced subgraph of $\Gammaone(G)$ whose vertices are elements of $\Omega$.

\subsection{The classification of finite simple groups}
Our results are all dependent on the classification of finite simple groups. The principal sources for information
on these groups and their covering groups is \cite{atlas} and \cite{kl}, which we have used very extensively, without necessarily
mentioning it explicitly in every instance. We have used the notation of \cite{atlas} for
finite simple and quasisimple groups, except in a few cases where we believe another usage is less likely to
cause confusion.

\begin{table}[!ht]
\centering
\begin{tabular}{lll}
Alternating groups      & $\Alt_n$ & $n\ge 5$\\ \\
Classical groups \\
\ \ \emph{Linear}       & $\PSL_n(q)$ & $n\ge 2$; not $\PSL_2(2)$ or $\PSL_2(3)$,\\
\ \ \emph{Unitary}      & $\PSU_n(q)$ & $n\ge 3$; not $\PSU_3(2)$,\\
\ \ \emph{Symplectic}   & $\PSp_{2m}(q)$ & $m\ge 2$; not $\PSp_4(2)$,\\
\ \ \emph{Orthogonal}   & $\POmega_{2m+1}(q)$ & $m\ge 3$; $q$ odd,\\
                        & $\POmega^+_{2m}(q)$ & $m\ge 4$,\\
                        & $\POmega^-_{2m}(q)$ & $m\ge 4$,\\ \\
Exceptional groups \\
\ \ \emph{Chevalley}    & $G_2(q),\ F_4(q),\ E_6(q)$; & not $G_2(2)$,\\
                        & $E_7(q),\ E_8(q)$, & \\
\ \ \emph{Steinberg}    & ${}^3D_4(q),\ {}^2E_6(q)$,\\
\ \ \emph{Suzuki}       & $\Sz(2^{2a+1})$ & $a\ge 1$,\\
\ \ \emph{Ree}          & ${}^2F_4(2^{2a+1})'$ & $a \ge 0$,\\
                        & ${}^2G_2(3^{2a+1})$ & $a\ge 1$\\ \\
Sporadic groups \\
\multicolumn{3}{l}{$\quad \Mat_{11},\ \Mat_{12},\ \Mat_{22}, \Mat_{23},\ \Mat_{24},\
                        \mathrm{J}_1,\  \mathrm{J}_2,\ \mathrm{J}_3,\ \mathrm{J}_4,\
                        \mathrm{Co}_3,\ \mathrm{Co}_2,\ \mathrm{Co}_1,\ \mathrm{Fi}_{22},$}\\
\multicolumn{3}{l}{$\quad \mathrm{Fi}_{23},\ \mathrm{Fi}_{24}',\ \mathrm{HS},\ \mathrm{McL},\ \mathrm{He},\
                        \mathrm{Ru},\ \mathrm{Suz},\ \mathrm{O'N},\ \mathrm{HN},\
                        \mathrm{Ly},\ \mathrm{Th},\ \mathrm{B},\ \mathrm{M}.$}
\end{tabular}
\caption{\label{t: FSG} Finite simple groups}
\end{table}

The non-abelian finite simple groups are listed in Table \ref{t: FSG}. In this table, and throughout the paper,
$q$ is a power of a prime $p$. The parameters in this list have been restricted in order to reduce the number
of occurrences of isomorphic groups under different names. The following isomorphisms remain:
\[
\Alt_5\cong \PSL_2(4)\cong \PSL_2(5),\
\PSL_2(7)\cong \PSL_3(2),\
\Alt_6\cong \PSL_2(9),\
\Alt_8\cong \PSL_4(2),\
\PSU_4(2)\cong \PSp_4(3).
\]


\subsection{Quasisimple groups}

Although the notation used in this paper is standard, we recall here some key definitions. A group is
{\it perfect} if it coincides with its commutator subgroup. (There is no connection between the usages of the word
`perfect' as it applies to graphs and to groups.) A {\it quasisimple group} is a perfect group $G$ such that
$G/Z(G)$ is simple.

All of the finite quasisimple groups are known, as a corollary to the Classification of Finite Simple Groups.
In most cases a quasisimple group $G$ has cyclic centre, and for a positive integer $n$ and
a simple group $S$, we write $n.S$ for a group $G$ such that $Z(G)$ is cyclic of order $n$ and
$G/Z(G) \cong S$. This notation extends in a natural way to groups with non-cyclic centres;  for instance we write
$(2 \times 2).S$ for a group with centre $C_2\times C_2$ and $G/Z(G)\cong S$.

In principle the notation just described does not specify a group up to isomorphism -- for instance,
there are may be several isomorphism classes of groups of type $n.S$ -- however
in all instances of the notation in this paper, the isomorphism class is in fact unique. This notation for
finite quasisimple groups is consistent with, for instance, \cite{atlas} and \cite{kl}.

A {\it component} of a group $G$ is a quasisimple subgroup $N$ which is {\it subnormal}, i.e.\
there exists a finite chain of subgroups of the form $N=N_0< N_1 < N_2 <\cdots < N_k =G$, such that $N_i$ is normal in
$N_{i+1}$ for all $i$. The significance of the set of components of $G$ has been demonstrated by
the seminal work of Bender, in which the notion of the {\it Generalized
Fitting Subgroup} $F^*(G)$ is defined (see, for instance, \cite{aschbacher3}).

\subsection{Basic lemmas}

\begin{lem}\label{l: subgroup obstruction}
If $H\leq G$ and $\Gammaone(H)$ is not perfect, then $\Gammaone(G)$ is not perfect.
\end{lem}
\begin{proof}
Any induced subgraph of $\Gammaone(H)$ is an induced subgraph of $\Gammaone(G)$. The result follows immediately.
\end{proof}

\begin{lem}\label{l: abelian centralizers}
 Suppose that $g$ is an element in $G$ for which $\Cent_G(g)$ is abelian. Then $g$ does not lie on a forbidden subgraph
 of $\Gammaone(G)$.
\end{lem}
\begin{proof}
The supposition implies that if any two elements $h,k$ are neighbours of $g$ in $\Gammaone(G)$, then there is an edge
between $h$ and $k$, and the result follows.
\end{proof}

\begin{lem}\label{l: central adjustment}
 Let $G$ be a group and let $g,h\in G$ be vertices of a forbidden subgraph of $\Gammaone(G)$. Then
 $gh^{-1}\notin Z(G)$.
\end{lem}
\begin{proof}
Suppose that $gh^{-1}\in Z(G)$. Then $\Cent_G(g)=\Cent_G(h)$, and so $g$ and $h$ have the
same neighbours in $\Gammaone(G)$. But for any two distinct vertices $u,v$ of a forbidden subgraph, there
is a third vertex $w$ of the subgraph which is connected to $u$ but not to $v$.
\end{proof}

The next lemma helps us to pass between simple groups and their quasisimple covers.
\begin{lem}\label{l: pass to qs}
Let $G$ be a group, let $Z$ be a central subgroup of $G$, let $K=G/Z$, and let ${\varphi:G\to K}$ be
the natural projection. Let $\Omega\subseteq K$ and suppose that for each $\omega\in\Omega$, the elements of
$\varphi^{-1}(\omega)$ are pairwise non-conjugate. Then $\Gammaone(\Omega)$ is perfect if and only if
$\Gammaone(\varphi^{-1}(\Omega))$ is perfect.
\end{lem}

\begin{proof}
For each $\omega \in\Omega$, let $\omega'$ be a pre-image
in $G$ of $\omega$, and let $\Omega'$ be the set $\{\omega' \mid \omega \in \Omega\}$.
It is clear that if two elements $h,\distinguish{h}\in \Omega'$ commute, then $\varphi(h)$ and
$\varphi(\distinguish{h})$ commute. On the other hand if $\varphi(h)$ and $\varphi(\distinguish{h})$ commute,
then $[\distinguish{h},h]\in Z$, and
so $h\distinguish{h}h^{-1}=z\distinguish{h}$ for some $z\in Z(G)$. But now the condition on pairwise non-conjugacy
implies that $z=1$, and
so $h$ and $\distinguish{h}$ commute. Hence $h$ and $\distinguish{h}$ commute if and only if
$\varphi(h)$ and $\varphi(\distinguish{h})$ commute,
and so we conclude that $\Gammaone(\Omega')\cong \Gammaone(\Omega)$.
Now $\Gammaone(\Omega')$ is an induced subgraph of $\Gammaone(\varphi^{-1}(\Omega))$, from which it follows that
if $\Gammaone(\varphi^{-1}(\Omega))$ is perfect then so is $\Gammaone(\Omega)$.

For the converse suppose that $\Gammaone(\varphi^{-1}(\Omega))$ contains a forbidden subgraph $\Delta$.
Lemma~\ref{l: central adjustment} tells us that the vertices of $\Delta$ have distinct images under $\varphi$.
Now $\Delta$ can be extended to a set $\Omega'$ as described above, and we have seen that
$\Gammaone(\Omega')\cong \Gammaone(\Omega)$. It follows that $\Gammaone(\Omega)$ contains a forbidden
subgraph, as required.
\end{proof}

Note that if $\omega$ is an element of $G$ with order coprime to $|Z(G)|$, then all elements of $\phi^{-1}(\omega)$ are pairwise non-conjugate.

\section{Commuting graphs of quasisimple groups}\label{s: proof}

In this section we study the commuting graphs of quasisimple groups.
We go through the various families given by the classification of finite simple groups, establishing which
groups have Berge graphs as their commuting graphs. These are precisely the perfect commuting graphs, by
the Strong Perfect Graph Theorem.

A technique we use frequently to show that a group $G$ has a non-perfect commuting graph,
is to exhibit a subgroup for which this is already known, and then invoke Lemma \ref{l: subgroup obstruction}.
Both the Strong Perfect Graph Theorem and Lemma
\ref{l: subgroup obstruction} will therefore be in constant use in this section. We shall usually suppress
explicit references to them, in order to avoid tedious repetitions.

In the cases where no subgroup of $G$ is known to have a non-perfect commuting graph, it is necessary
to determine whether $\Gamma(G)$ contains a forbidden subgraph directly. We use a variety of techniques for exhibiting
odd length cycles. In most cases these have length $5$,
but we have made no particular effort to describe the shortest cycle possible. Indeed, for the infinite
families $2\cprod \Alt_n$ and ${}^2G_2(q)$ our arguments yield $7$-cycles in the commuting graphs, although it is known
that $5$-cycles exist in almost all of the groups in these families. In fact we know of only three finite
quasisimple groups $G$ such that $\Gamma(G)$ is non-perfect, but contains no induced subgraph
isomorphic to a $5$-cycle; these are
\[
2\cprod \Alt_7,\ \PSL_2(13),\ \PSL_2(17),
\]
each of which can be shown to have a $7$-cycle as a forbidden subgraph.
We believe that there exist no further
examples, but we have not attempted to prove this.

\subsection{Alternating groups}

\begin{lem}\label{l: S5}
 Let $G=\Sym_5$, the symmetric group on $5$ letters. Then $\Gammaone(G)$ is not perfect.
\end{lem}
\begin{proof}
 The induced subgraph on the vertices $(1\, 2), (2\,3), (3\,4), (4\,5), (1\,5)$ is a $5$-cycle.
\end{proof}

\begin{lem}\label{l: alternating}
 Let $G=\Alt_n$, the alternating group on $n$ letters. Then $\Gammaone(G)$ is perfect if and only if $n\leq 6$.
\end{lem}
\begin{proof}
The centralizer of every non-identity element of $\Alt_5$ is abelian, and so $\Gammaone(\Alt_5)$ is perfect by
Lemma~\ref{l: abelian centralizers}. On the other hand, if $n\geq 7$ then $\Alt_n$ has a subgroup isomorphic to
$\Sym_5$, and so Lemma~\ref{l: S5} tells us that $\Gammaone(\Alt_n)$ is not perfect.

It remains to  deal with the case $n=6$. The only non-trivial elements of $\Alt_6$
with non-abelian centralizers have order $2$, and so if $\Lambda$
is a forbidden subgraph of $\Gammaone(G)$, then every vertex of $\Lambda$ is an involution. Let $g\in G$ be a vertex of
$\Lambda$.
Observe that $\Cent_G(g)\cong D_8$, the dihedral group of order $8$. Since there are only five involutions in $D_8$ we
conclude that $\Lambda$ is either a cycle of odd order, or else the complement of a $7$-cycle. (Recall that the
complement of a $5$-cycle is another $5$-cycle.)

Suppose that $\Lambda$ is the complement of a $7$-cycle. We may assume that one of its vertices is
$(1\,2)(3\,4)$. The neighbours of this vertex in $\Lambda$ can only be the other four involutions in its
centralizer, namely
\[
(1\,3)(2\,4),\ (1\,2)(5\,6),\ (1\,4)(2\,3),\ (3\,4)(5\,6).
\]
Similarly the neighbours of $(1\,4)(2\,3)$ must be
\[
(1\,2)(3\,4),\ (1\,3)(2\,4),\ (1\,4)(5\,6),\ (2\,3)(5\,6).
\]
But we have now listed seven involutions (not including repetitions), and it is easily checked that the
induced subgraph on these vertices is not the complement of a $7$-cycle; this is a contradiction.

We have still to show that $\Lambda$ cannot be a cycle of odd order at least $5$.
The group $G$ has two conjugacy classes of subgroups isomorphic to $C_2\times C_2$. Exactly one
subgroup from each class is contained in $\Cent_G(g)$; let these subgroups be $A$ and $B$.
If $h$ and $h'$ are the neighbours of $g$ in $\Lambda$, then
since $h$ and $h'$ do not commute, we see that the subgroups $\langle g,h\rangle$ and $\langle g,h'\rangle$ are
distinct; so one of them is $A$ and the other $B$. It follows that if we colour each edge of $\Lambda$ according
to whether the vertices it connects generate a conjugate of $A$ or a conjugate of $B$, then we have a $2$-colouring of
the edges of $\Lambda$. But this is a contradiction, since a cycle of odd length is not $2$-colourable.
\end{proof}

\begin{cor}\label{c: symmetric}
Let $G=\Sym_n$, the symmetric group on $n$ letters. Then $\Gammaone(G)$ is perfect if and only if $n\leq 4$.
\end{cor}
\begin{proof}
 Lemma~\ref{l: S5} implies that $\Gammaone(G)$ is not perfect if $n\geq 5$. On the
 other hand $\Alt_6$ has a subgroup isomorphic to $\Sym_4$ and so Proposition~\ref{p: alternating} implies that $\Gammaone(\Sym_4)$ is perfect; hence the same is true of
 $\Gammaone(\Sym_n)$ with $n<4$.
\end{proof}

Now we generalize Lemma~\ref{l: alternating} to deal with quasisimple covers of alternating groups.

\begin{prop}\label{p: alternating}
Let $G$ be a quasisimple group such that $G/Z(G)\cong \Alt_n$ for some $n\geq 5$. Then $\Gammaone(G)$ is perfect if
and only if $G$ is equal to one of the groups in the following list:
\[
\Alt_5,\ \Alt_6,\ 2\cprod \Alt_5,\ 2\cprod \Alt_6,\ 3\cprod \Alt_6,\ 6\cprod \Alt_6,\ 6\cprod \Alt_7.
\]
\end{prop}
\begin{proof}
If $G$ is simple then it is either $\Alt_5$ or $\Alt_6$ by Lemma~\ref{l: alternating}, and so we may
suppose that $G$ is not simple.
If $n\geq 7$ and $|Z(G)|=2$, then the induced
subgraph in $\Alt_n$ on the vertices
\[
(1\,2\,3),\ (4\,5\,6),\ (1\,2\,7),\ (3\,4\,5),\ (1\,6\,7),\ (2\,3\,4),\ (5\,6\,7)
\]
is a $7$-cycle. Since all of these elements have order $3$, while $Z(G)$ has order $2$, each element lifts to
elements of order $3$ and $6$ in $G$. Now the lifts commute exactly when their projective images commute,
and so the induced subgraph of $\Gammaone(G)$ on the lifts of order $3$ is a $7$-cycle.

Since the Schur multiplier of $\Alt_n$ has order $2$ for $n\geq 8$, we may now suppose that
$n\leq 7$. We deal with the remaining groups one by one.

The groups $2\cprod \Alt_5, 2\cprod \Alt_6$ and $6\cprod \Alt_6$ are AC-groups (being isomorphic to $\SL_2(5)$,
$\SL_2(9)$ and $3\cprod \SL_2(9)$ respectively), and so have perfect commuting graphs by
Lemma~\ref{l: abelian centralizers}.

In $3\cprod \Alt_6$, the only non-central
elements with non-abelian centralizers are elements whose square is central (i.e.\ they are lifts of involutions in
$\Alt_6$). No forbidden subgraph
can contain two vertices in the same coset of the centre, as these would have the same set of neighbours.
So we may restrict our attention to the commuting graph on the involutions of $3\cprod \Alt_6$.
But this commuting graph is isomorphic to the commuting graph on the involutions of $\Alt_6$ (see Lemma~\ref{l: pass to qs}), which we have already
seen to be perfect.

To exclude $3\cprod \Alt_7$ we recall that $\Gammaone(\Alt_7)$ contains a $5$-cycle whose vertices are involutions.
Each of these involutions lifts to a unique involution in $3\cprod \Alt_7$ and the induced subgraph of
$\Gammaone(3\cprod \Alt_7)$ on these involutions is, again, a $5$-cycle.

Finally suppose that $G=6\cprod \Alt_7$. There
are two conjugacy classes of non-abelian subgroups which arise as centralizers in $G$. Let $T$ be the set of elements
with non-abelian centralizers and $\Gamma(T)$ the induced subgraph of $\Gamma(G)$ with vertices in $T$. It is a
straightforward computation that if $g,h\in G$ are conjugate elements of $T$ which commute, then their centralizers are
equal. Let ${\Gamma(T)/\sim}$  be the quotient of $\Gammaone(G)$ obtained by identifying vertices with the same
centralizer. Then this graph is bipartite, and hence perfect. It follows easily that $\Gammaone(G)$ is perfect.
\end{proof}

\subsection{Linear groups of dimension \texorpdfstring{$2$}{2}}

\begin{lem}\label{l: sl2}
 Suppose that $G\leq\GL_2(q)$. Then $\Gammaone(G)$ is perfect.
\end{lem}
\begin{proof}
 If $g$ is a non-central element of $G$, then $\Cent_G(g)$ is abelian. Now the result follows from
 Lemma~\ref{l: abelian centralizers}.
\end{proof}

Lemma~\ref{l: sl2} implies, in particular, that the quasisimple groups $\SL_2(q)$ have perfect commuting graphs. The
next result deals with most of the remaining $2$-dimensional quasisimple groups.

\begin{lem}\label{l: psl2 first}
 If $G=\PSL_2(q)$ with $q$ odd and $q>9$, then $\Gammaone(G)$ is not perfect.
\end{lem}
\begin{proof}
Define $\epsilon$ by
\[
\epsilon = \left\{\begin{array}{ll} 1 & \textrm{if $q\equiv 1 \bmod 4$,} \\ -1 & \textrm{if q $\equiv 3 \bmod 4$.}
\end{array}\right.
\]
The group $G$ has a single class $T$ of involutions, of size $q(q+\epsilon)/2$. This is the only conjugacy class of $G$ whose elements have non-abelian centralizers. The centralizer of each involution is a
dihedral group of order
$q-\epsilon$.

The graph $\Gammaone(T)$ is a regular graph of degree
$(q-\epsilon)/2$. Let $t\in T$, and let $\Omega_d$ be the set of vertices in $\Gammaone(T)$ at distance $d$ from $t$.
Then $|\Omega_1|= (q-\epsilon)/2$. Each vertex $s$ in $\Omega_1$ is connected to exactly one other vertex in $\Omega_1$,
and so $s$ has $(q-\epsilon-4)/2$ neighbours in $\Omega_2$.

We claim that $\Gammaone(T)$ contains no subgraph (induced or otherwise) isomorphic to a $4$-cycle.
To prove this claim, let us suppose that $T'=\{t_1,t_2,t_3,t_4\}$ is a subset of $T$
such that $\Gamma(T')$ contains a $4$-cycle, with vertices in the order listed.
Recall that the intersection of the centralizers of distinct involutions in $\PSL_2(q)$ is abelian,
being a subgroup of a Klein $4$-group. Since each of $t_2$ and $t_4$ centralizes both $t_1$ and $t_3$, we see
that $t_2$ and  $t_4$ commute. So $\Gamma(T')$ is not a $4$-cycle, and the claim is proved.

An immediate corollary to the claim is that if $\Gammaone(T)$ contains a $5$-cycle as a subgraph,
then that subgraph is an induced subgraph. The claim implies, moreover, that no two vertices in $\Omega_1$ have a common neighbour in $\Omega_2$,
and so we have
\[
|\Omega_2| = \frac{1}{4}(q-\epsilon)(q-\epsilon-4).
\]

Each vertex $r$ in $\Omega_2$ has a unique neighbour $s\in \Omega_2$ such that $r$ and $s$ have a common neighbour in
$\Omega_1$. Suppose that $r$ has another neighbour $u\in \Omega_2$. If $r'$ and $u'$ are the neighbours of $r$ and $u$,
respectively, in $\Omega_1$, then it is clear that the induced subgraph on the vertices $t,\ r',\ r,\ u,\ u'$ is a $5$-cycle.

We may suppose, then, that any two neighbours in $\Omega_2$ have a common neighbour in $\Omega_1$, and hence no common
neighbour in $\Omega_3$. Let $x\in\Omega_3$, and suppose that $r$ and $s$ are neighbours of $x$ in $\Omega_2$. Then $r$
and $s$ are not adjacent, and have no common
neighbour in $\Omega_1$. Let $r'$ and $s'$ be the neighbours of $r$ and $s$, respectively, in $\Omega_1$. If $r'$ and
$s'$ are adjacent, then
the induced subgraph on $\{r',r,x,s,s'\}$ is a $5$-cycle, and so again here, $\Gammaone(G)$ is not perfect.

We may therefore assume that for any vertex $x$ in $\Omega_3$, and for any pair $u,v$ of neighbours in $\Omega_1$, there
exists at most one $r\in\Omega_2$
such that $r$ is joined to $x$ and to either of $u$ or $v$. It follows that the number of neighbours for $x$ in
$\Omega_2$ cannot be greater than $|\Omega_1|/2=(q-\epsilon)/4$. Now each
element of $\Omega_2$ has $(q-\epsilon-4)/2$ neighbours in $\Omega_3$, and so we have
\[
|\Omega_3| \ge \frac{2(q-\epsilon-4)}{q-\epsilon}|\Omega_2| = \frac{1}{2}(q-\epsilon-4)^2.
\]
Now clearly $|T|\ge 1+|\Omega_1|+|\Omega_2|+|\Omega_3|$; but asymptotically the left-hand side of this inequality is
$q^2/2$, whereas the right-hand side is $3q^2/4$. We may therefore bound $q$ above; specifically, we obtain the
inequality
\[
q^2-(18+8\epsilon)q +(39+18\epsilon) \le 0.
\]
When $\epsilon=-1$ this implies that $q\le 7$, and for $\epsilon=1$ that $q\le 17$.

It remains only to check the cases $q=13$ and $q=17$, which require separate treatment since
$\Gamma(\PSL_2(q))$ does not have a $5$-cycle as an induced subgraph in these cases.
A straightforward computation shows that in each group there exists an involution $t$, and an element $g$ of order
$(q+1)/2$, such that $[t,t^g]=1$. Now it is not hard to show that the induced subgraph on the conjugates $t,t^g,t^{g^2},\ldots$ is a cycle of order $(q+1)/2$.
\end{proof}

We see that if $q$ is even or at most $9$, then $\Gammaone(\PSL_2(q))$ is perfect.
For even $q$ this follows immediately from Lemma \ref{l: sl2}. For $\PSL_2(5)\cong \Alt_5$ and $\PSL_2(9)\cong \Alt_6$, see
Proposition \ref{p: alternating}, and for $\PSL_2(7)\cong \SL_3(2)$ see Lemma \ref{l: sl34}
below. (The commuting graph of the non-quasisimple group $\PSL_2(3)$ is easily seen to be perfect.)

The exceptional covers of $\PSL_2(9)$ have also been dealt with in Proposition \ref{p: alternating} above. We
thus have a complete classification of those quasisimple groups $G$ such that
$G/Z(G)\cong\PSL_2(q)$ and such that $\Gammaone(G)$ is perfect.

\begin{rem}\label{r: pgl}
It follows from Lemmas~\ref{l: sl2} and \ref{l: psl2 first} that the commuting graph of $\PGL_2(q)$ is not perfect
when $q$ is odd and $q>9$. We remark that $\Gammaone(\PGL_2(q))$ is not perfect when $q\in \{5,7,9\}$ either;
we omit the proof of this fact, which is straightforward to establish computationally.
\end{rem}

\subsection{Classical groups of dimension 3}

\begin{lem}\label{l: sl3} Let $G$ be isomorphic to $\SL_3(q)$ or $\PSL_3(q)$ with $q\neq 2,4$.
Then the commuting graph of $G$ is not perfect.
\end{lem}

The proof that follows shows, in addition, that the commuting graph of $\GL_3(q)$ is non-perfect for $q>2$.

\begin{proof}
Let $\alpha$ and $\beta$ be distinct non-zero elements of $\Fq$. Then the five matrices
\[
\left(\begin{array}{ccc}1&0&0\\0&1&0\\0&1&1\end{array}\right),
\left(\begin{array}{ccc}1&1&0\\0&1&0\\0&0&1\end{array}\right),
\left(\begin{array}{ccc}1&0&1\\0&1&0\\0&0&1\end{array}\right),
\]\[
\left(\begin{array}{ccc}1&0&0\\0&1&1\\0&0&1\end{array}\right),
\left(\begin{array}{ccc}\alpha&0&0\\0&\beta&0\\0&0&\beta\\\end{array}\right)
\]
constitute a $5$-cycle subgraph of the commuting graph of $\GL_3(q)$. For the last of these matrices to lie in $\SL_3(q)$
we
require that $\alpha\beta^2=1$; this equation is soluble (by distinct elements $\alpha,\beta$) unless $q$ is $2$ or $4$.

It remains only to observe that the images of these matrices in $\PSL_3(q)$ induce a $5$-cycle in $\Gammaone(\PSL_3(q))$.
\end{proof}

\begin{lem}\label{l: sl34}
Let $G$ be isomorphic to one of $\SL_3(2)$, $\SL_3(4)$ or $\PSL_3(4)$. Then $\Gammaone(G)$ is perfect.
\end{lem}
\begin{proof}
It suffices to prove the lemma for $\SL_3(4)$ and $\PSL_3(4)$, since $\SL_3(2)$ is contained in each as a subgroup.

The only non-central elements of $\SL_3(4)$ whose centralizers are non-abelian are the transvections. These have order
$2$ or $6$, and fall in three conjugacy classes. These classes merge into one class of involutions in $\PSL_3(4)$, and
this class contains all of the non-trivial elements
of $\PSL_3(4)$ with non-abelian centralizers. From these facts, it easily follows that $\Gammaone(\SL_3(4))$ and
$\Gammaone(\PSL_3(4))$ are perfect if and only if the induced subgraph $\Gamma(T)$, common to both, on the set $T$ of
involutory transvections, is perfect.

Let $V$ be the natural module for $\SL_3(4)$. For a transvection $t$, we write $H(t)$ for the hyperplane of $V$ fixed by
$t$. We write $L(t)$ for the image of $t-I$ (equivalently, the unique
$1$-dimensional $t$-invariant subspace $\langle v\rangle$ of $V$ such that $t$ acts trivially on $V/\langle v\rangle$).
It is easy to show that elements $t$ and $u$ of $T$ commute if and only if either $H(t)=H(u)$ or $L(t)=L(u)$.

Let $S$ be a subset of $T$ of size at least $5$. Suppose that $S$ contains distinct vertices $t$ and $u$ such that
$H(t)=H(u)$ and $L(t)=L(u)$. Then for each $v\in S$ distinct from $t$ and $u$, we see that $v$ is adjacent to $t$ if
and only if it is adjacent to $u$. It follows that the induced subgraph of
$\Gammaone(\SL_3(4))$ on $S$ cannot be a cycle or its complement. We shall therefore suppose that $S$ contains no such
elements $t$ and $u$. Now we may colour each edge $(t,u)$ of the induced subgraph of $\Gamma(T)$ on $S$ with colours $H$
and $L$, depending on whether $H(t)=H(u)$ or $L(t)=L(u)$.

Suppose that the induced subgraph on $S$ is a cycle $(t_1,\dots,t_k)$. For $t_i\in S$, we see that the colour of
$(t_i,t_{i-1})$ is not the colour of $(t_i,t_{i+1})$, or else  $t_{i-1}$ and $t_{i+1}$ would commute. It follows
immediately that the cycle is $2$-colourable, and hence that it has even length.

Suppose on the other hand that the induced subgraph on $S$ is the complement of a $k$-cycle. We may assume that $k>5$, since
the complement of a $5$-cycle is another $5$-cycle. Let $t_1,t_2$ be two connected vertices in $S$ -- so either $H(t_1)=H(t_2)$ or else $L(t_1)=L(t_2)$. Now $\{t_1,t_2\}$ is a subset of $k-5$ distinct triangles in $\Gamma(S)$ and so at least $k-3$ vertices in $\Gamma(S)$ have the same colour. But this implies that $S$ contains a complete subgraph on $k-3$ vertices, which is impossible, since $k\ge 7$ by assumption, and the complement of a $k$-cycle has no clique of size greater than $k/2$.
\end{proof}

By reference to \cite[Table 5.1.D]{kl} we see that the results given above attend to almost all
quasisimple covers of the groups
$\PSL_3(q)$. When $q=2$ the group $\PSL_3(2)$ has an exceptional cover isomorphic to
$\SL_2(7)$; however this has been dealt with in the previous section and so can be excluded here.
The remaining exceptions occur when $q=4$, and the next lemma deals with this situation.

\begin{lem}\label{l: l34}
Suppose that $G$ is a quasisimple group such that $G/Z(G)\cong \PSL_3(4)$. Then $\Gammaone(G)$ is perfect if and only
if $G$ is one of the groups in the following list.
\[
\PSL_3(4),\ 2\cprod \SL_3(4),\ 3\cprod \PSL_3(4),\ (2\times 2)\cprod \PSL_3(4),\ 6\cprod \PSL_3(4),\
\]
\[
(6\times 2)\cprod \PSL_3(4),\
(4\times 4)\cprod \PSL_3(4),\ (12\times 4) \cprod \PSL_3(4).
\]
\end{lem}
\begin{proof}
By reference to \cite[Table 5.1.D]{kl}, we observe that the Schur multiplier of $\PSL_3(4)$ is
$C_{12}\times C_4$. The extension $3\cprod\PSL_3(4)$ is isomorphic to $\SL_3(4)$. The elements of $\PSL_3(4)$
have orders from the set $\{1,2,3,4,5,7\}$; there is a unique conjugacy class $T$ of involutions.

Let $G$ be a quasisimple extension of $\PSL_3(4)$. If $g$ is an element of $G$ whose image in $\PSL_3(4)$ has order
$3$,$4$,$5$ or $7$, then the centralizer of $g$ in $G$ is abelian. For each involution $t\in\PSL_3(4)$, let
$t_G$ be an element of $G$ which projects onto $t$, and let
$T_G=\{t_G\mid t \in T\}$. Then Lemma~\ref{l: central adjustment} implies that $\Gammaone(G)$ is perfect if and only if its induced subgraph
$\Gamma(T_G)$ on $T_G$ is perfect.

Since the graphs $\Gamma(T_G)$, for the various quasisimple extensions $G$, all have vertex sets in
natural bijection to one another, we can represent them all using a single ornamented graph.
Let $M=(12\times 4)\cprod \PSL_3(4)$ be the full covering group.
Let $\Gamma(T)$ be the commuting graph on $T$. We endow $\Gamma(T)$ with an edge-labelling, where the label of the
edge $(s,t)$ is determined by the commutator $[s_M,t_M]$, an element of $Z(M)$. In fact only four labels are needed,
since if $s$ and $t$ are commuting elements of $\PSL_3(4)$ then $[s_M,t_M]$ has order at most $2$ \cite{GAP}, and hence lies in the
unique subgroup $V$ of $Z(M)$ isomorphic to $C_2\times C_2$.

Each quasisimple group $G$ is a central quotient of $M$, and it is clear that the commuting graph $\Gamma(T_G)$ is
determined by the image $V_G$ of $V$ in this quotient. If $V_G$ is trivial, then $\Gamma(T_G)$ is
isomorphic to $\Gamma(T)$, which by Lemma \ref{l: sl34} is perfect. If $V_G\cong C_2\times C_2$ then $T_G\cong T_M$,
which consists of $105$ connected components, each isomorphic to the triangle graph $K_3$; so clearly $\Gamma(T_G)$ is
perfect in this case also.

For the remaining cases, recall that the elements of $T$ correspond to transvections in $\SL_3(4)$, and
that each transvection has associated with it a hyperplane $H(t)$ of fixed points, and a line $L(t)$, which is the
image of $I-t$.
Transvections $s$ and $t$ commute if and only if $H(s)=H(t)$ or $L(s)=L(t)$. For any hyperplane $H$ in $\F_4^3$
there are
$15$ transvections $t$ such that $H(t)=H$, and for each line $L$ there are $15$ such that $L(t)=L$. Furthermore, for
each pair $(H,L)$ such that $L< H$, there are three transvections $t$ such that $H(t)=H$ and $L(t)=L$ (which yield the
triangles in $\Gamma(T_M)$ described above). There
are $21$ lines and $21$ hyperplanes in $\F_4^3$, and so $\Gamma(T)$ may be expressed as a union of
$42$ copies of the complete graph $K_{15}$.

The vertices of each of these copies of $K_{15}$ generate an elementary abelian group $A$ of order $16$. The group $A$
is naturally endowed with the structure of a $2$-dimensional vector
space over the field $\F_4$, scalar multiplication
being given by the rule $(\lambda,t)\mapsto I+\lambda(t-I)$ for $\lambda\in \F_4$.

We are now in a position to deal with groups $G$ such that $V_G\cong C_2$. Let $v\in V$ be the non-identity element of
the kernel of the map $V\to V_G$. We associate the group $V$ with the additive group $\F_4$,
with $I$ as $0$ and with $v$ as $1$. It is not hard to show that the map
$(s,t) \mapsto [s_M,t_M]$ defines a non-degenerate alternating form on $A$. Let $s$ and $t$ be a hyperbolic pair
with respect to this form; so $(s,t)=v=1$. Let $\alpha$ be a primitive element of $\F_4$, and
consider the induced subgraph of $\Gamma(T)$ on the vertices
\[
s,\ \alpha s,\ \alpha^{-1}s+\alpha^{-1}t,\ \alpha t,\ t.
\]
We see that in the order listed above, edges between consecutive vertices receive labels $0$ or $1$, whereas
other edges receive labels $\alpha$ or $\alpha^{-1}$. It  follows that these vertices induce a $5$-cycle in $\Gamma(T_G)$,
and so $\Gammaone(G)$ is not perfect.

Thus the commuting graph of $G$ is perfect if and only if $G\cong M/Z_0$ where $Z_0\leq Z(M)$ and either $V\leq Z_0$ or $V\cap Z_0=\{1\}$. It is an easy matter to ascertain which groups $Z_0\leq Z(M)$ satisfy this condition and one obtains quotients as listed. Note that for some of these quotients, $2.\PSL_3(4)$ for instance, there is more than one choice for the subgroup $Z_0$; in such cases we can appeal to \cite[Theorem 6.3.2]{gls3} to see that they are all isomorphic.
\end{proof}

\begin{lem}\label{l: unitary3}
Let $G$ be a quotient of $\SU_3(q)$ by a central subgroup, where $q>2$. Then $\Gammaone(G)$ is not perfect.
\end{lem}
\begin{proof}
For details about the dilatation and transvection mappings used in this argument, we refer the reader to
\cite[Chapter 2]{Dieu}.

Suppose first that $G=\SU_3(q)$. Let $V$ be the natural module for $G$, and let $F$ be the underlying Hermitian form on $V$.
For any $1$-dimensional subspace $L$ of $V$, there is a non-central element $X$ of $G$ such that $L$ is $X$-invariant,
and such that $X$ acts as a scalar on $L^{\perp}$.

If $L$ is non-singular with respect to $F$, then $X$ is a scalar multiple of a
dilatation in $\GU_3(q)$, with axis $L$ and centre $L^\perp$. The transformation $X$ may be chosen to have
order $2$ if $q$ is odd, or order $q+1$ if $q$ is even.
The centralizer of $X$ in $G$ is equal to the stabilizer of $L$.

On the other hand if $L$ is singular with respect to $F$ then $X$ is a scalar multiple of a transvection, again with
axis $L$ and centre $L^\perp$.
In this case $X$ may be chosen to have order $p$, where $p$ is the characteristic. In this case the centralizer of $X$
in $G$ is a proper subgroup of the stabilizer of $L$, of index $q^2-1$.

Let $\Omega$ be a set of transformations $X$ of the types described above, one for each line in $V$. We write $L(X)$ for the axis of $X$.
Suppose that $X,Y \in \Omega$ be distinct elements which commute. Then $Y$ stabilizes $L(X)$, and
since $L(X)\ne L(Y)$, it is easy to see that $L(X) \in L(Y)^\perp$. Conversely, suppose that $L(X)\in L(Y)^\perp$;
then $L(X)$ and $L(Y)$ cannot both be singular, since $V$ has no totally singular $2$-dimensional subspace. We may
suppose without loss of generality that $L(X)$ is non-singular; now since
$Y$ acts as a scalar on $L(Y)^\perp$ we see that $Y$ is in the stabilizer of $L(X)$, which is equal to the centralizer
of $X$.

Let $\Delta_F$ be the graph whose vertices are $1$-dimensional subspaces of $V$, with edges connecting lines which are
perpendicular with respect to $F$. Then we have shown that $\Delta_F$ is isomorphic to the subgraph of $\Gammaone(G)$
induced on the vertices $\Omega$. Let $(v_1,v_2,v_3)$ be a basis
for $V$; we may take $F$ to be the hermitian form given by
\[
F(v_i,v_j)=\left\{\begin{array}{ll} 1 & \textrm{if $(i,j) = (1,1)$, $(2,3)$ or $(3,2)$,}\\ 0& \textrm{otherwise.}\end{array}\right.
\]
Now it is clear that the set of lines containing the points
\[
v_1,\ v_2,\ v_1+v_2,\ v_1-v_3,\ v_3
\]
induces a $5$-cycle in $\Delta_F$ and we are done.

Now suppose that $G=\SU_3(q)/A$, where $A$ is a central subgroup of $\SU_3(q)$.
Let $X$ and $Y$ be in $\Omega$. It is straightforward to check that
the images of $X$ and $Y$ in $\PSU_3(q)$ commute if
and only if $X$ and $Y$ commute.
It follows immediately that $\Gammaone(G)$ is not perfect.
\end{proof}

\subsection{Classical groups of dimension at least \texorpdfstring{$4$}{4}}

We start with a general result for all classical groups of large enough dimension over almost all fields.

\begin{lem}\label{l: higher rank}
Let $q$ be a prime power with $q\neq 2,4$. Let $G$ be a quasisimple classical group with $G/Z(G)$ isomorphic to
$\PSL_n(q)$ or $\PSU_n(q)$ with $n\ge 4$, or to $\PSp_{2m}(q)$, $\POmega^\pm_{2m}(q)$ or $\POmega_{2m+1}(q)$
with $m\ge 3$. Then $\Gammaone(G)$ is not perfect.
\end{lem}
\begin{proof}
If $G/Z(G)\not\cong \PSU_n(q)$, then it is a standard result that $G$ contains a parabolic subgroup $P$ for which a
Levi complement $L$ contains a normal subgroup isomorphic to either $\PSL_3(q)$ or $\SL_3(q)$. Next suppose that $G/Z(G)\cong \PSU_n(q)$. The group $\SU_n(q)$ contains a subgroup $H_0$ that stabilizes a non-degenerate subspace of dimension $3$; now let $H_1$ be the lift of $H_0$ in the universal version of $\PSU_n(q)$ (in all cases except $(n,q)=(4,3)$, this universal version is just $\SU_n(q)$ itself and so $H_1=H_0$) and let $H$ be the projective image of $H_1$ in $G$. Then $H$ contains a normal
subgroup isomorphic to either $\PSU_3(q)$ or $\SU_3(q)$ (see, for instance, \cite[\S\S4.1 and 4.2]{kl}).

The result now follows from Lemmas~\ref{l: sl3} and \ref{l: unitary3}.
\end{proof}

We now work through the families of classical groups one by one; the force of Lemma \ref{l: higher rank} is that
we have only to deal with the case that $q$ is $2$ or $4$, and with groups $G$ such that $G/Z(G)\cong \PSp_4(q)$.

\begin{prop}\label{p: symp}
Let $G$ be a quasisimple group with $G/Z(G)$ isomorphic to $\PSp_{2m}(q)$, with $m\geq 2$.
Then $\Gammaone(G)$ is not perfect.
\end{prop}
\begin{proof}
By Lemma \ref{l: higher rank}, it will be sufficient to deal with the case that $q$ is even, and with the
case $m=2$.

Suppose that $q$ is even. Since $\Sp_4(2)\cong \Sym_6$, we know from Corollary~\ref{c: symmetric} that the
commuting graph of $\Sp_4(2)$ is not perfect. Since $\Sp_{2m}(q)$ has $\Sp_4(2)$ as a subgroup for $m\ge 2$,
it follows that $\Gammaone(\Sp_{2m}(q))$ is not perfect.
Referring to \cite[Table 5.1.D]{kl}, we see that the only quasisimple group left to consider is the double cover
of $\Sp_6(2)$. But reference to \cite{atlas} shows that this group contains $\PSU_3(3)$ as a subgroup, and hence it
has a non-perfect commuting graph by Lemma~\ref{l: unitary3}.

Suppose next that $q$ is odd and that $m=2$.
Referring to \cite{atlas} we see that $\PSp_4(3)$ contains a subgroup isomorphic to $\Sym_6$.
If $q>3$, then \cite[Proposition 4.3.10]{kl} tells us that $G$ contains a field extension subgroup isomorphic
to $\PSL_2(q^2).2$; the commuting graph of this subgroup is not perfect by Lemma \ref{l: psl2 first}. It follows that,
in either case, $G$ has a non-perfect commuting graph.

It remains to deal with the groups $\Sp_4(q)$ for odd $q$. Our argument is similar to that for the
unitary groups $\PSU_3(q)$ in
Lemma~\ref{l: unitary3} above, and we again refer the reader to \cite[Chapter 2]{Dieu} for facts about transvections.
Let $G=\Sp_4(q)$, and let $V$ be the natural module for $G$, with $F$ the underlying alternating form on $V$. For any
non-zero $v\in V$, the transvection map $T_v:x\mapsto x+F(x,v)v$ lies in $G$. The maps $T_v$ and $T_w$ commute if and
only if $F(v,w)=0$. Let $\{e_1,f_1,e_2,f_2\}$ be a hyperbolic basis for $V$; so $F(e_1,f_1)=F(e_2,f_2)=1$, and
$\langle e_2,f_2\rangle = \langle e_1,f_1\rangle^\perp$. Define
\[
v=e_1,\quad  w= e_2,\quad x=f_1,\quad y=f_1+f_2,\quad z= e_1-e_2+f_2.
\]
Then the induced subgraph of $\Gammaone(G)$ on the vertices $T_v,\ T_w,\ T_x,\ T_y,\ T_z$ is a $5$-cycle, and so
$\Gammaone(G)$ is not perfect.
\end{proof}

\begin{prop}\label{p: orth}
Let $G$ be a quasisimple group with $G/Z(G)$ isomorphic to $\POmega_{2m+1}(q)$ with $m\ge 3$ or to
$\POmega^\pm_{2m}(q)$ with $m\ge 4$. Then $\Gammaone(G)$ is not perfect.
\end{prop}
\begin{proof}
By Lemma \ref{l: higher rank}, it is sufficient to deal with the case that $q$ is even.
Furthermore, we may suppose that $G/Z(G)\cong \POmega^\pm_{2m}(q)$, because of the isomorphism
$\POmega(2m+1,2^k)\cong \PSp(2m,2^k)$.
Now \cite[Proposition 4.1.7]{kl} implies that $G$ contains a subgroup isomorphic to a quasisimple cover of
$\Sp_{n-2}(q)$, and the result follows from Proposition~\ref{p: symp}.
\end{proof}

\begin{prop}\label{p: lin un even}
 Let $n\geq 4$ and let $G$ be a quasisimple group with $G/Z(G)$ isomorphic to $\PSL_n(q)$ or $\PSU_n(q)$. Then
 $\Gammaone(G)$ is not perfect.
\end{prop}
\begin{proof}
We suppose that $q$ is even, since otherwise Lemma~\ref{l: higher rank} gives the result.

Suppose that $n$ is even. Then \cite[Propositions 4.5.6 and 4.8.3]{kl} imply that $\Sp_{n}(q)$ is a subgroup of
both $\PSL_{n}(q)$ and $\PSU_n(q)$. If $(n,q)\neq (4,2)$, then $\Sp_n(q)$ is simple and
so some quasisimple cover of $\Sp_n(q)$ is a subgroup of $G$. The result now follows from Proposition~\ref{p: symp}.
If $G/Z(G)\cong \PSL_4(2)$, then the result follows from
Proposition~\ref{p: alternating} since $\PSL_4(2) \cong \Alt_8$. If $G/Z(G)\cong\PSU_4(2)$,
then the result follows from Proposition~\ref{p: symp}, since $\PSU_4(2)\cong \PSp_4(3)$.

Suppose, on the other hand, that $n$ is odd. If $G=\PSL_n(q)$, then $G$ contains a subgroup $H$ such that
$H/Z(H)\cong \PSL_{n-1}(q)$. Similarly,
if $G=\PSU_n(q)$, then $G$ contains a subgroup $H$ such that $H/Z(H)\cong \PSU_{n-1}(q)$.
Since $n-1$ is even, the result in each case follows from above.
\end{proof}

\subsection{Ree and Suzuki groups}

\begin{prop}\label{p: suzuki}
 If $G=\Sz(q)$ with $q>2$, then $\Gammaone(G)$ is perfect.
\end{prop}
The result is also true for $q=2$, but we omit it from the statement since $\Sz(2)$ is not simple.

\begin{proof}
 We refer to \cite{suzuki2} and observe that the only non-trivial elements in $G$ which have non-abelian centralizer
 are the involutions (and there is a single conjugacy class of these). Let $\Lambda$ be a forbidden subgraph of
 $\Gammaone(G)$ and observe that all of its vertices correspond to involutions in $G$. Let $g$ be one such. Then
 the set of involutions which commute with $g$ lie in an elementary abelian subgroup of $\Cent_G(g)$ and hence any two
 neighbours of $g$ in $\Lambda$ must themselves be neighbours, a contradiction.
\end{proof}

\begin{prop}\label{p: ree 1}
 If $G={^2F_4}(q)'$ with $q\geq 2$, then $\Gammaone(G)$ is not perfect.
\end{prop}
\begin{proof}
We consult \cite{atlas} to see that ${^2F_4}(2)'$ contains $\Sym_6$ and hence, by Lemma~\ref{l: S5}, ${^2F_4}(2)'$ is not perfect. Since $G$ contains ${^2F_4}(2)'$ as a subgroup, the result follows.
\end{proof}

\begin{prop}\label{p: ree 2}
 If $G={^2G_2}(q)$, then $\Gammaone(G)$ is not perfect.
\end{prop}
\begin{proof}
We first deal with the case $q=3$, when $G$ is not quasisimple, but isomorphic to the automorphism group of
$\SL_2(8)$. Let $F$ be the automorphism of $\SL_2(8)$ induced by the field automorphism
$x\mapsto x^2$. Let $\alpha$ be an element of $\F_8$ such that $\alpha^3+\alpha=1$. We define the following elements
of $\SL_2(8)$:
\[
J=\left(\begin{array}{cc}1&0\\1&1\end{array}\right),\quad K=\left(\begin{array}{cc}1&1\\0&1\end{array}\right),\quad
X=\left(\begin{array}{cc}\alpha&0\\\alpha^6&\alpha^6\end{array}\right),\quad
Y=\left(\begin{array}{cc}\alpha^4&\alpha\\0&\alpha^3\end{array}\right).
\]
Now it is a straightforward computation to verify that the induced subgraph of $\Gammaone(G)$ on the vertices
\[
F^X,\ J^X,\ J,\ F,\ K,\ K^Y,\ F^Y
\]
is isomorphic to a $7$-cycle, and so $\Gammaone(G)$ is not perfect. (The point of
this construction is that the matrices $X$ and $Y$ lie in opposite Borel subgroups in $\SL_2(8)$, and that the
commutator $[XY^{-1},F]$ is fixed by $F$. It is perhaps worth noting that there is no induced subgraph of
$\Gammaone(G)$ isomorphic to a $5$-cycle in this case.)

We now observe that if $q>3$ then $G$ contains a subgroup isomorphic to ${^2G_2}(3)$, and so
the result follows.
\end{proof}

There are two non-simple quasisimple groups whose quotients are Ree or Suzuki groups and we deal with these in the
final result of this section.

\begin{lem}\label{l: suzuki exceptional}
 If $G=2\cprod \Sz(8)$ or $(2\times 2)\cprod\Sz(8)$, then $\Gammaone(G)$ is perfect.
\end{lem}
\begin{proof}
 Using Magma \cite{Magma} we establish that $G$ has precisely one non-central conjugacy class $C$ of involutions.
 What is more $C$ is the only non-central conjugacy class whose members have non-abelian centralizers. Thus, by
 Lemma~\ref{l: abelian centralizers}, it is enough to show that $\Gammaone(C)$ is perfect.

Let $g\in C$ and suppose that $\Lambda$ is a forbidden subgraph of $\Gammaone(C)$. The set of involutions which
commute with $g$ lie in an elementary abelian subgroup of $\Cent_G(g)$ and hence any two neighbours of $g$ in $\Lambda$
must themselves be neighbours, a contradiction.
\end{proof}

\subsection{The remaining exceptional groups}

\begin{prop}\label{p: g2}
If $G$ is a quasisimple group with $G/Z(G)$ isomorphic to $G_2(q)$ or to ${^3D_4}(q)$,
then $\Gammaone(G)$ is not perfect.
\end{prop}
\begin{proof}
Suppose first that $G$ is simple. Referring to \cite{kleidman3} we see that $G_2(q)<{^3D_4}(q)$ for all $q$.
Furthermore \cite{atlas} and \cite{kleidman2} imply that $\PSU_3(3) = G_2(2)' < G_2(q)$ for all $q$, and the result follows from Lemmas~\ref{l: subgroup obstruction} and \ref{l: unitary3}.

If $G$ is not simple, then $G=2\cprod G_2(4)$ or $3\cprod G_2(3)$. In both cases $G$ contains a subgroup
isomorphic to $\PSU_3(3)$ and the result follows as before.
\end{proof}

\begin{prop}\label{p: the rest}
Let $G$ be a quasisimple group with $G/Z(G)$ isomorphic to one of $F_4(q)$, ${^2E_6}(q)$, $E_6(q)$, $E_7(q)$ or $E_8(q)$.
Then $\Gammaone(G)$ is not perfect.
\end{prop}
\begin{proof}
Referring to \cite{liebecksaxl} we see that
\[
{^3D_4}(q) < F_4(q) < {E_6}(q),\ {^2E_6}(q).
\]
Furthermore, the
\emph{universal} version of $E_6(q)$ is a subgroup of the \emph{adjoint} version of $E_7(q)$, and likewise the universal version of 
$E_7(q)$ is a subgroup of the adjoint version of $E_8(q)$. Since the Schur multiplier of ${^3D_4}(q)$ is trivial 
we conclude that all quasisimple
covers of the (simple) adjoint versions of $F_4(q), {^2E_6}(q), {E_6}(q), E_7(q)$ and $E_8(q)$ contain a subgroup
isomorphic to ${^3D_4}(q)$, and the result follows from Proposition~\ref{p: g2}.
\end{proof}

\subsection{Sporadic groups}

\begin{prop}\label{p: sporadic}
 If $G$ is a quasisimple group with $G/Z(G)$ isomorphic to a sporadic simple group, then $\Gammaone(G)$ is not perfect.
\end{prop}
\begin{proof}

Our strategy here is to find, for each sporadic simple group, a subgroup which has already been shown to have
non-perfect commuting graph. The result will then follow from Lemma~\ref{l: subgroup obstruction}.
Our essential reference is \cite{atlas}, which provides lists of maximal subgroups of these groups.
For reasons of transparency, we use only subgroup inclusions which are immediately visible from the structural
information these lists provide (though the subgroups need not themselves be maximal).

We deal first with the simple groups. We have the following subgroup inclusions.
\begin{eqnarray*}
   \Sym_5 &<& \Mat_{11},\ \Mat_{12},\ \mathrm{Th},\ \mathrm{B},\ \mathrm{M},\\
   \Alt_7 &<& \Mat_{22},\ \Mat_{23},\ \Mat_{24},\ \mathrm{HS},\ \mathrm{McL},\ \mathrm{Co}_1,\
   \mathrm{Fi}_{23},\ \mathrm{Fi}_{24}',\ \mathrm{O'N},\\
   \Alt_8 &<& \mathrm{Ru}, \\
   \Mat_{12} &<& \mathrm{Suz},\ \mathrm{Fi}_{22},\ \mathrm{HN},\\
   \Mat_{23} &<& \mathrm{Co}_3,\ \mathrm{Co}_2,\\
   \Mat_{24} &<& \mathrm{J}_4,\\
   \PSL_2(11) &<& \mathrm{J}_1,\\
   \PSL_2(19) &<& \mathrm{J}_3,\\
   \PSU_3(3) &<& \mathrm{J}_2,\\
   \Sp_4(4) &<& \mathrm{He},\\
    G_2(5) &<& \mathrm{Ly}.
\end{eqnarray*}

We deal now with the case that $G$ is non-simple, the following subgroup inclusions cover most possibilities.
 \begin{eqnarray*}
   \Mat_{11} &<& 2\cprod \Mat_{12},\ 2\cprod \mathrm{HS},\ 3\cprod \mathrm{McL},\ 3\cprod \mathrm{O'N}, \\
   \PSU_3(3) &<& 2\cprod \mathrm{J}_2, \\
   \PSL_2(19) &<& 3\cprod \mathrm{J}_3, \\
   {^2F_4}(2)' &<& 2\cprod \mathrm{Ru},\ 2\cprod \mathrm{Fi}_{22},\ 3\cprod \mathrm{Fi}_{22},\
   6\cprod \mathrm{Fi}_{22},\ 2\cprod \mathrm{B},\\
   \mathrm{Co}_2 &<& 2\cprod \mathrm{Co}_1, \\
   \mathrm{Fi}_{23} &<& 3\cprod\mathrm{Fi}_{24}'.
 \end{eqnarray*}
In addition all quasisimple covers of $\mathrm{Suz}$ contain a quasisimple cover of $G_2(4)$.

We are left with the possibility that $G/Z(G)\cong \Mat_{22}$. Note that the simple group
$\Mat_{22}$ contains a subgroup
isomorphic to $\Alt_7$ and all quasisimple covers of $\Alt_7$ have non-perfect commuting graph, except $6\cprod \Alt_7$.
Thus, for $\Gammaone(G)$ to be perfect, the subgroup $\Alt_7$ in $\Mat_{22}$ must lift to a subgroup
$6\cprod \Alt_7$ in $G$. This implies immediately that $|Z(G)| =6$ or $12$. Now we consult
\cite{atlas} to see that elements of order $4$ in $\Mat_{22}$ do not lift to elements of order $24$ when
$|Z(G)|=6$. Since elements of order $4$ in $\Alt_7$ lift to elements of order $24$ in $6\cprod \Alt_7$
we conclude that $6\cprod \Mat_{22}$ does not contain $6\cprod \Alt_7$. Thus we must have $G=12\cprod \Mat_{22}$.

Now we refer to \cite[Table 1]{hhm}, to see that a maximal subgroup of
$\Mat_{22}$ which is isomorphic to $\PSL_3(4)$ lifts in $12\cprod \Mat_{22}$ to a cover whose centre is cyclic and has order divisible by $4$.
All such covers of $\PSL_3(4)$ have non-perfect commuting graph and the result
follows by Lemma~\ref{l: l34}.

\end{proof}

\section{Components in finite groups}\label{s: general}

In this section we prove Theorem~\ref{t: general} and Corollary~\ref{c: ac}.
The next result is required for the proof of Theorem~\ref{t: general}, and
also illustrates a diagrammatic method we have found helpful.

\begin{prop}\label{p: at most 2 components}
\begin{enumerate}
\item Let $K$, $L$ and $M$ be finite non-abelian groups. Then $\Gammaone(K\times L\times M)$ is not perfect.
\item Let $K,L,M$ be three distinct finite non-abelian subgroups of a group $G$, each of which centralizes the other two.
Then $\Gammaone(G)$ is not perfect.
\end{enumerate}
\end{prop}
\begin{proof}
\begin{enumerate}
\item Define $(k,\distinguish{k})$, $(\ell,\distinguish{\ell})$ and $(m,\distinguish{m})$ to be pairs of non-commuting elements from
$K$, $L$ and $M$ respectively. Now the five elements
\begin{equation}\label{e: 5-cycle}
(1,l,m), \, (k', 1, 1), \, (1, l', 1), \, (k,1,m'), \, (k,l,1),
\end{equation}
induce a 5-cycle in $\Gamma(K\times L\times M)$ and we are done.

\item If $K$, $L$ and $M$ are subgroups of $G$ which centralize one another,
then there is a natural homomorphism $K\times L\times M\to G$ given by $(x,y,z)\mapsto xyz$.
It is easy to check that the images under this map of the five elements constructed in part (i),
induce a $5$-cycle in $\Gammaone(G)$.
\end{enumerate}
\end{proof}

Before we proceed, let us take a moment to understand more clearly why the elements listed at \eqref{e: 5-cycle}
induce a 5-cycle. To do this we refer to Figure~\ref{f: 2components} in which we draw the commuting graphs of the
three projections of the listed tuples. Note that we maintain the same orientation for each graph, so that the
vertex corresponding to the entry from the first tuple is at the `east' of the graph, and entries from the following
tuples are written anticlockwise
around the graph. Now it is clear that the commuting graph of the three $5$-tuples listed at \eqref{e: 5-cycle}
has edges between two vertices precisely when all three projections have edges between the corresponding
vertices. This observation immediately implies that the tuples listed at \eqref{e: 5-cycle} form a $5$-cycle,
as required. In the arguments below we shall use the same convention for representing projections.

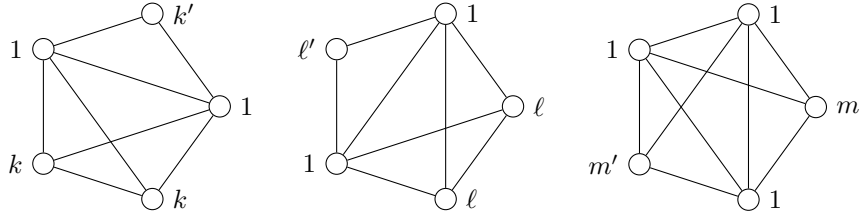
\begin{figure}[!ht]
\centering
\begin{tabular}{ccc}

\newcount\mycount

\begin{tikzpicture}[scale=0.24]
  \node[draw,circle,inner sep=0.10cm, label={right:$1$}] (N-1) at (0:5.4cm) {};
  \node[draw,circle,inner sep=0.10cm, label={right:$\distinguish{k}$}] (N-2) at (72:5.4cm) {};
 \node[draw,circle,inner sep=0.10cm, label={left:$1$}] (N-3) at (144:5.4cm) {};
 \node[draw,circle,inner sep=0.10cm, label={left:$k$}] (N-4) at (216:5.4cm) {};
 \node[draw,circle,inner sep=0.10cm, label={right:$k$}] (N-5) at (288:5.4cm) {};
 \path (N-1) edge (N-2);
  \path (N-2) edge (N-3);
  \path (N-3) edge (N-4);
  \path (N-4) edge (N-5);
  \path (N-5) edge (N-1);
\path (N-1) edge (N-3);
\path (N-3) edge (N-5);
\path (N-4) edge (N-1);
      \end{tikzpicture}

    &
  \newcount\mycount
\begin{tikzpicture}[scale=0.24]
  \node[draw,circle,inner sep=0.10cm, label={right:$\ell$}] (N-1) at (0:5.4cm) {};
  \node[draw,circle,inner sep=0.10cm, label={right:$1$}] (N-2) at (72:5.4cm) {};
 \node[draw,circle,inner sep=0.10cm, label={left:$\distinguish{\ell}$}] (N-3) at (144:5.4cm) {};
 \node[draw,circle,inner sep=0.10cm, label={left:$1$}] (N-4) at (216:5.4cm) {};
 \node[draw,circle,inner sep=0.10cm, label={right:$\ell$}] (N-5) at (288:5.4cm) {};
 \path (N-1) edge (N-2);
  \path (N-2) edge (N-3);
  \path (N-3) edge (N-4);
  \path (N-4) edge (N-5);
  \path (N-5) edge (N-1);
\path (N-2) edge (N-4);
\path (N-5) edge (N-2);
\path (N-4) edge (N-1);
      \end{tikzpicture}

    &
  \newcount\mycount
\begin{tikzpicture}[scale=0.24]
  \node[draw,circle,inner sep=0.10cm, label={right:$m$}] (N-1) at (0:5.4cm) {};
  \node[draw,circle,inner sep=0.10cm, label={right:$1$}] (N-2) at (72:5.4cm) {};
 \node[draw,circle,inner sep=0.10cm, label={left:$1$}] (N-3) at (144:5.4cm) {};
 \node[draw,circle,inner sep=0.10cm, label={left:$\distinguish{m}$}] (N-4) at (216:5.4cm) {};
 \node[draw,circle,inner sep=0.10cm, label={right:$1$}] (N-5) at (288:5.4cm) {};
 \path (N-1) edge (N-2);
  \path (N-2) edge (N-3);
  \path (N-3) edge (N-4);
  \path (N-4) edge (N-5);
  \path (N-5) edge (N-1);
\path (N-1) edge (N-3);
\path (N-2) edge (N-4);
\path (N-3) edge (N-5);
\path (N-5) edge (N-2);
      \end{tikzpicture}

\end{tabular}
 \caption{Three projections needed for Proposition~\ref{p: at most 2 components} (i)}
\label{f: 2components}
 \end{figure}

Our next result is in similar vein and to state it we need some terminology: We say that $\Gammaone(G)$ {\it contains
a $4$-chain} if there is an induced subgraph of $\Gammaone(G)$ isomorphic to a path graph on four vertices. 

\begin{prop}\label{p: 4chain use}
Let $K$ and $L$ be subgroups of a group $G$ such that $\Gammaone(K)$ contains a $4$-chain, $L$ is non-abelian, and
$K$ and $L$ centralize one another. Then $\Gammaone(G)$ is not perfect.
\end{prop}
\begin{proof}
\begin{figure}[!ht]
\centering
\begin{tabular}{cc}
\begin{tikzpicture}[scale=0.3]
  \node[draw,circle,inner sep=0.10cm, label={right:$k_4$}] (N-1) at (0:5.4cm) {};
  \node[draw,circle,inner sep=0.10cm, label={right:$1$}] (N-2) at (72:5.4cm) {};
 \node[draw,circle,inner sep=0.10cm, label={left:$k_1$}] (N-3) at (144:5.4cm) {};
 \node[draw,circle,inner sep=0.10cm, label={left:$k_2$}] (N-4) at (216:5.4cm) {};
 \node[draw,circle,inner sep=0.10cm, label={right:$k_3$}] (N-5) at (288:5.4cm) {};
 \path (N-1) edge (N-2);
  \path (N-2) edge (N-3);
  \path (N-3) edge (N-4);
  \path (N-4) edge (N-5);
  \path (N-5) edge (N-1);
\path (N-2) edge (N-4);
\path (N-5) edge (N-2);
      \end{tikzpicture}

      &

\begin{tikzpicture}[scale=0.3]
  \node[draw,circle,inner sep=0.10cm, label={right:$1$}] (N-1) at (0:5.4cm) {};
  \node[draw,circle,inner sep=0.10cm, label={right:$\distinguish{\ell}$}] (N-2) at (72:5.4cm) {};
 \node[draw,circle,inner sep=0.10cm, label={left:$1$}] (N-3) at (144:5.4cm) {};
 \node[draw,circle,inner sep=0.10cm, label={left:$\ell$}] (N-4) at (216:5.4cm) {};
 \node[draw,circle,inner sep=0.10cm, label={right:$\ell$}] (N-5) at (288:5.4cm) {};
 \path (N-1) edge (N-2);
  \path (N-2) edge (N-3);
  \path (N-3) edge (N-4);
  \path (N-4) edge (N-5);
  \path (N-5) edge (N-1);
\path (N-1) edge (N-3);
\path (N-3) edge (N-5);
\path (N-4) edge (N-1);
      \end{tikzpicture}

\end{tabular}
 \caption{\label{f: five}Two projections needed for Proposition~\ref{p: 4chain use}}
 \end{figure}
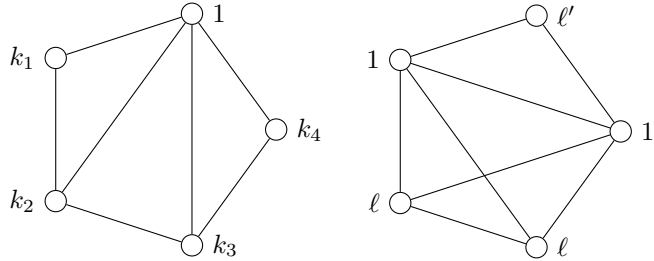

Let $k_1,k_2,k_3,k_4$ be the vertices of a $4$ chain in $K$, and let $\ell,\distinguish{\ell}$ be
non-commuting elements in $L$. Then the projection graphs in Figure~\ref{f: five} illustrate that the five elements
\[
k_1,\ k_2\ell,\ k_3\ell,\ k_4,\ \distinguish{\ell}
\]
in $KL$ induce a $5$-cycle in $\Gammaone(G)$.
\end{proof}

\begin{lem}\label{l: 4chain}
\begin{enumerate}
\item If $G$ is isomorphic to one of the groups in the following list, then $G$ contains a $4$-chain:
\[
\Alt_6,\ 3\cprod \Alt_6,\ 6\cprod \Alt_7,\ \PSL_3(2),\ \PSL_3(4),\ 2\cprod\PSL_3(4),\ 3\cprod\PSL_3(4),\
(2\times 2)\cprod\PSL_3(4),\ 6\cprod\PSL_3(4),\ (6\times 2)\cprod\PSL_3(4).
\]
\item If $G$ is a quasisimple group such that $\Gammaone(G)$ is
perfect and contains a $4$-chain, then $G$ is one of the groups listed in (i).
\end{enumerate}
\end{lem}
\begin{proof}
\begin{enumerate}
\item If $G\cong \Alt_6$, we can take $g_1=(1\,5)(3\,4), g_2=(1\,5)(2\,6), g_3=(1\,2)(5\,6), g_4=(1\,2)(3\,4)$. If
$G\cong 3\cprod \Alt_6$, then we can take pre-images of these four elements.
If $G=6\cprod \Alt_7$ then we can take pre-images in $G$ of
\[
g_1=(1\,2\,3\,4)(5\,6), g_2=(1\,3)(2\,4), g_3=(5\,6\,7), g_4=(1\,2)(3\,4)(5\,6\,7).
\]
If $G=\SL_3(2)$ then we can take
\[
   g_1=\left(\begin{matrix}
             1 & 0 & 1 \\
             0 & 1 & 1 \\
             0 & 0 & 1
             \end{matrix}\right),\
   g_2=\left(\begin{matrix}
             1 & 0 & 1 \\
             0 & 1 & 0 \\
             0 & 0 & 1
             \end{matrix}\right),\
   g_3=\left(\begin{matrix}
             1 & 1 & 1 \\
             0 & 1 & 0 \\
             0 & 0 & 1
             \end{matrix}\right),\
   g_4=\left(\begin{matrix}
             1 & 0 & 0 \\
             0 & 0 &  1\\
             0 & 1 & 0
             \end{matrix}\right).
\]

For the remaining cases we refer back to the proof of Lemma \ref{l: l34}. It is clearly sufficient to consider the
induced subgraph $\Gamma(T_G)$ introduced there. In fact this graph is the same for any of the extensions of $\PSL_3(4)$
listed here (since the central elementary abelian $2$-subgroup $A$ of the full covering group $M$ of $\PSL_3(4)$ is
contained in the kernel of the quotient homomorphism $M\to G$ in each case.) It is sufficient, therefore, to
find a $4$-chain in any one of these groups. Since $3\cprod\PSL_3(4)$ is isomorphic to $\SL_3(4)$, it contains
$\SL_3(2)$ as a subgroup, and so the four elements given above for $\SL_3(2)$ can be used in this case also.

\item For this part we must show that every group listed in Theorem~\ref{t: main}
but not in Lemma~\ref{l: 4chain} does not contain a $4$-chain. This is obviously the case for all of the
AC-groups, which are listed in Corollary~\ref{c: ac}.

If $G$ is $(4\times 4)\cprod\PSL_3(4)$ or $(12\times 4)\cprod\PSL_3(4)$, and if a $4$-chain existed in $\Gammaone(G)$,
then there would be a $4$-chain in the graph $\Gamma(T_G)$ constructed in the proof of Lemma \ref{l: l34}. But it was
seen in that proof that $\Gamma(T_G)$ is a union of pairwise disconnected triangles,
and so clearly no $4$-chain exists there.

Finally, suppose that $G$ is equal either $\Sz(q)$ for some $q=2^{2n+1}$,
or else to $2\cprod \Sz(8)$ or $(2\times 2)\cprod\Sz(8)$.
In any of these cases $G$ has a single class of non-central elements with non-abelian centralizers, consisting of
involutions. Furthermore, for each involution $g$ in this class,
the involutions commuting with $g$ generate an elementary abelian subgroup of $G$.
From these facts it is clear that $\Gammaone(G)$ can have no $4$-chain.
\end{enumerate}
\end{proof}

\begin{proof}[Proof of Theorem~\ref{t: general}]
Proposition~\ref{p: at most 2 components} tells us that $G$ has at most two components, and
Lemma~\ref{l: subgroup obstruction} implies that the commuting graphs of the components are perfect, and so
each component of $G$ are isomorphic to one of the quasisimple groups listed in Theorem~\ref{t: main}.

Suppose first that $G$ has a unique component $N$ and that case~(i) of the theorem does not hold. Then $G$ appears in the list of Lemma~\ref{l: 4chain}, and so $\Gammaone(G)$ contains a $4$-chain. It follows from Proposition~\ref{p:
4chain use} that no non-abelian subgroup of $G$ can centralize $N$, and so $\Cent_G(N)$ is abelian and (ii) holds.

Next suppose that $G$ has two components $N_1$ and $N_2$. Since $N_1$ and $N_2$ centralize one another,  and
since they are both non-abelian, it follows from Proposition~\ref{p: 4chain use} that neither contains a
$4$-chain.  and that $\Gammaone(N_1)$ contains a $4$-chain. Therefore each is isomorphic to one of the groups
listed in case~(i) of the theorem.

Let $C=\Cent_G(N_1N_2)$. Then $N_1$, $N_2$ and $C$ are three subgroups of $G$ which centralize one another, and
it follows from Proposition~\ref{p: at most 2 components} that one of them is abelian. Since $N_1$ and $N_2$ are
quasisimple, we see that $C$ is abelian and (iii) holds.
\end{proof}

\begin{proof}[Proof of Corollary~\ref{c: ac}]
Lemma~\ref{l: abelian centralizers} implies that any finite AC-group $G$ has a perfect commuting graph.
If $G$ is quasisimple, then $G$ is one of the groups listed in Theorem~\ref{t: main}. It is easy to check
that $\SL_2(q)$ and $6\cprod \Alt_6$ are AC-groups.

We observe that a centreless AC-group has abelian Sylow $p$-subgroups for all $p$. Now $\Sz(2^{2a+1})$,
$\Alt_6$, $\SL_3(2)$ and $\PSL_3(4)$ have non-abelian Sylow $2$-subgroups, and hence they are not AC-groups.
It is not hard to check that if $G$ is isomorphic to $3\cprod \Alt_6$, to a quasisimple cover of $\PSL_3(4)$, or to
a quasisimple cover of $\Sz(q)$, then the centralizer of a non-central involution is non-abelian.
In $6\cprod \Alt_7$ the centralizer
of an element of order $4$ is non-abelian. This is sufficient to prove (i).

To establish (ii), let $G$ be an arbitrary finite AC-group. Since components are non-abelian, and since any two
distinct components centralize one another, it is clear that $G$ must have a unique component $N$.
It is also clear that $N$ must itself be an AC-group, and so $N$ is one
of the groups listed in part (i).
Let $C=\Cent_G(N)$, and let $Z=Z(G)$. Suppose that $g\in C\backslash Z$. Then $N\le \Cent_G(g)$, and since $N$ is
non-abelian we have a contradiction. Hence $\Cent_G(N) = Z$, and so $G/Z$ is isomorphic to a subgroup of $\Aut(N)$.

Suppose first that $N=\SL_2(q)$, and that $G$ contains an element $g$ whose action on $N/Z(N)$ induces a
field automorphism. Then $\Cent_{N/Z(N)}(g)\cong \PSL_2(q_0)$, where $q=q_0^a$ for some $a>1$. If $q$ is even
then $Z(N)$ is trivial and, since $\PSL_2(q_0)$ is non-abelian, we immediately obtain a contradiction. If $q$ is odd, then $|Z(N)|=2$,
and hence $\Cent_N(g)$ contains a subgroup isomorphic to a subgroup of $\PSL_2(q_0)$ of index at most $2$. Once again we
conclude that $\Cent_N(g)$ is non-abelian, which is a contradiction. So $G/Z(G)$ contains no element acting as a field
automorphism, and we conclude that $G/Z(G)$ is isomorphic either to $\PSL_2(q)$, or to an extension of
$\PSL_2(q)$ of degree $2$. So we see that $|G:NZ(G)|\le 2$ in this case.

Suppose next that $N=6\cprod \Alt_6$. We refer to \cite[Table 6.3.1]{gls3}, which asserts that the action of
$\Out(\Alt_6)$ on $Z(N)$ is non-trivial. So if $G/Z\cong \Aut(\Alt_6)$ then $G$ contains an element $g$ which
acts non-trivially on $Z(N)$. Thus not all non-trivial elements of $Z(N)$ are central in $G$. But since all
non-trivial elements of $Z(N)$ have non-abelian centralizer, this is a contradiction. So $G/Z(G)$ is a proper
subgroup of $\Aut(\Alt_6)$, and since $|\Out(\Alt_6)|=4$ we have $|G:NZ(G)|\le 2$ in this case too.
\end{proof}

\section{Improvements}\label{s: extensions}

Improvements on Theorems~\ref{t: main} and \ref{t: general} are certainly possible, and in this final section we discuss some possibilities. 

\subsection{Almost quasisimple groups} An obvious first step would be to extend Theorem~\ref{t: main} to classify almost quasisimple groups with perfect commuting graphs. We recall that an {\it almost quasisimple group} is a group with a single component $N$ and, furthermore, this component $N$ is quasisimple. 

It is an easy matter to use Theorem~\ref{t: main} to write down the almost quasisimple groups that are candidates for having a perfect commuting graph. To do this efficiently we need the notion of isoclinism. Recall, first, the definition of the commutator map:
\[
 [-,-]: G\times G \to G, (x,y) \mapsto x^{-1}y^{-1}xy.
\]
Clearly, we can think of the commutator map as being a function of form $G/Z(G) \times G/Z(G) \to G'$. Two groups $G$ and $H$ are said to be {\it isoclinic} if there are two isomorphisms $\varphi:G/Z(G)\to H/Z(H)$ and $\theta: G'\to H'$ that commute with the two commutator maps, i.e. the following diagram commutes:
\begin{center}
\begin{tikzcd}
G/Z(G)\times G/Z(G) \arrow{r}{[-,-]} \arrow{d}[swap]{\varphi\times\varphi} & G' \arrow{d}{\theta} \\
H/Z(H)\times H/Z(H) \arrow{r}{[-,-]} & H'
\end{tikzcd}
\end{center}

If a group $G$ is perfect, then $G$ has a unique central extension $M.G$ realizing any quotient $M$ of its Schur multiplier. Thus for example there is a unique proper cover $2.A_n$ of the alternating group $A_n$, for $n\geq 5$. For more general groups $G$, the appropriate groups $M.G$ are only unique up to isoclinism. Thus, for example, in our discussion below there may be several almost quasisimple groups $3.S_6$ that we should consider. The next lemma asserts that considering one is enough.

\begin{lem}
 If $G$ and $H$ are isoclinic, then $\Gammaone(G)$ is perfect if and only if $\Gammaone(H)$ is perfect.
\end{lem}
\begin{proof}
 Let $\varphi:G/Z(G)\to H/Z(H)$ and $\theta: G'\to H'$ be the relevant isomorphisms. Suppose that $\Gammaone(G)$ is not perfect, and let $\Lambda=\{g_1,\dots, g_k\}$ be a subset of $G$ such that the induced subgraph on $\Lambda$ is an odd cycle. Now, for $i=1,\dots, k$, let $h_i\in H$ be such that $h_iZ(H)=\varphi(g_i Z(G))$. 
 
 Now observe that, since $\theta^{-1}(1_H)=1_G$, we conclude that $[g_i, g_j]=1$ if and only if $[h_i, h_j]=1$. Thus $\{h_1,\dots, h_k\}$ is an odd cycle and $\Gammaone(H)$ is not perfect.
 
 The same argument with $G$ and $H$ swapped, and $\varphi$ and $\theta$ replaced by $\varphi^{-1}$ and $\theta^{-1}$, proves the converse.
\end{proof}

We are now in a position to list those almost quasisimple groups that may have a perfect commuting graph, and we do this in Table~\ref{t: quasisimple}. Clearly, if an almost quasisimple group is to have perfect commuting graph, then its quasisimple normal subgroup must also have perfect commuting graph, hence the table is broken down into rows according to Theorem~\ref{t: main}. Groups in the central column are prescribed up to isoclinism.

\begin{table}[!htbp]
\centering
\begin{tabular}{lll}
\hline
Component  & Group & Comments \\
\hline
$\PSL_3(2)$ &  $\PGL_2(7)$ & {\bf Non-perfect} by Remark~\ref{r: pgl}. \\
\hline
$A_6$ & $S_6$ & {\bf Non-perfect} by Corollary~\ref{c: symmetric}. \\
& $\PGL_2(9)$ &  {\bf Non-perfect} by Remark~\ref{r: pgl}. \\
& $M_{10}$&  {\bf Perfect}; see comments below. \\
\hline
 $3.A_6$ &  $3.S_6$ &  {\bf Non-perfect} by \cite{GAP}. \\
 & $3.\PGL_2(9)$ & {\bf Non-perfect} by \cite{GAP}. \\
 & $3.M_{10}$ & {\bf Perfect}; see comments below.\\
\hline
$6.A_6$.
&  $6.S_6$ & {\bf Non-perfect} by \cite{GAP}. \\
& $6.\PGL_2(9)$ & {\bf Perfect}; see comments below. \\
& $12.M_{10}$ & {\bf Perfect}; see comments below. \\
\hline
$6.A_7$ & $6.S_7$ & {\bf Non-perfect} by \cite{GAP}. \\
\hline
$2.\Sz(8)$ & None & \\
\hline 
$(2\times 2).\Sz(8)$ & $(2\times 2).\Sz(8).3$ & {\bf Non-perfect} by \cite{GAP}.\\
\hline
${\rm X}.\PSL_3(4)$ & Various & {\bf Inconclusive.} \\
\hline
$\SL_2(q)$ & Various & {\bf Inconclusive.} \\
\hline
$\Sz(q)$ & Various & {\bf Inconclusive.}
\end{tabular}

\caption{\label{t: quasisimple} Almost quasisimple groups whose commuting graph may be perfect}
\end{table}

Some comments about Table~\ref{t: quasisimple} are in order. Note, first, that in the case where an almost quasisimple group $G$ exists, and we have already listed a subgroup of $G$ with non-perfect commuting graph, then $G$ does not appear in the Table~\ref{t: main}. So, for instance, $\PGammaL_2(9)$ is an almost quasisimple group with unique component $A_6$. Since $A_6$ appears in Theorem~\ref{t: main}, we should study $\PGammaL_2(9)$. However $\PGammaL_2(9)$ contains a subgroup isomorphic to $S_6$ which, as we see in the table, has a  non-perfect commuting graph. Hence we may omit $\PGammaL_2(9)$ from the list.

Note, second, that the row starting ${\rm X}.\PSL_3(4)$ references all almost quasisimple groups whose unique component is a quasisimple cover of $\PSL_3(4)$ that occurs in Theorem~\ref{t: main}. Full facts in this situation are unknown, although we remark that $\Gammaone(\PGL_3(4))$ is non-perfect by (an adaptation of) Lemma~\ref{l: sl3}. We also remark that the final two rows of Table~\ref{t: quasisimple} refer to infinite families of groups.

Finally we should justify the assertions in the final column: In the cases where we have written ``{\bf Non-perfect} by \cite{GAP}'', we mean that we have run computations in GAP and found cycles of odd order in the commuting graph of the given group. \footnote{To do this we have made use of presentations found in the online ATLAS of Finite Group Representations \cite{atlas2}; where presentations have not been available in \cite{atlas2}, we have received assistance from Professor J\"urgen M\"uller for which we would like to record our very sincere thanks.}

The cases where the commuting graph is perfect require more explanation: Let $X\in\{1,3\}$; then $\Gammaone(X.M_{10})$ is perfect because $\Gammaone(X.A_6)$ is perfect and all elements in $M_{10}\setminus A_6$ have abelian centralizers. Similarly $\Gammaone(12.M_{10})$ is perfect because it contains a subgroup $K\cong C_2\times 6.A_6$ which is isoclinic to $6.A_6$; now $\Gammaone(6.A_6)$ is perfect and all elements in $12.M_{10}\setminus K$ have abelian centralizers, and we are done. The graph $\Gammaone(6.\PGL_2(9))$ is perfect because, as computations in \cite{GAP} confirm, the graph contains neither $5$-cycles nor $5$-chains.

\subsection{Extending Theorem~\ref{t: general}} In the absence of a full classification of those almost quasisimple groups that have perfect commuting graph, we will not write down a theorem extending Theorem~\ref{t: general}. Instead, we offer the following result which pertains to a specific situation, and which illustrates the leverage that extra information about almost quasisimple groups can bring.

\begin{prop}\label{p: a6}
 Let $G$ be a finite group such that $\Gammaone(G)$ is perfect, and suppose that $G$ has a component $N$ isomorphic to
 $\Alt_6$. Let $C=\Cent_G(N)$. Then $C$ is abelian, and the quotient group $G/C$ is isomorphic either to
 $\Alt_6$ or to the Mathieu group $\Mat_{10}$.
\end{prop}
\begin{proof}
We observe first that $G$ must fall under case (ii) of Theorem~\ref{t: general}, and so $N$ is the unique
component of $G$, and $C$ is abelian. Since $G/NC$ is isomorphic to a subgroup of $\Out(N)$, we see that
$G/C$ is an almost simple group with socle $N$. Reference to \cite{atlas} tells us that
$G/C$ is isomorphic to one of $\Alt_6$, $\Sym_6$, $\Mat_{10}$ or $\PGL_2(9)$, or to the
projective semilinear group $\mathrm{P}\Gamma\mathrm{L}_2(9)$ which contains all of the others as subgroups.

If $G=NC$ then $G/C\cong \Alt_6$. So we suppose that $G\ne NC$.
Then there exists $g\in G$ is such that $gNC$ has order $2$ in the quotient $G/NC$.
Let $H=\langle N,g\rangle$, and observe that $HC/C$ is an almost simple group of order $2|N|$.
Since $H/N$ is cyclic, we have $[h_1,h_2]\in N$ for all $h_1,h_2\in H$, and now since
$N\cap C$ is trivial, it follows that any two conjugate elements of $H$ lie in distinct cosets of $H\cap C$.
Therefore the conjugacy action of $H$ on its normal subgroup $H\cap C$ is trivial, and so
$H\cap C$ is central in $H$.

Now $H$ is a subgroup of $G$, and so $\Gammaone(H)$ is perfect; so Lemma~\ref{l: pass to qs} tells us that
$\Gammaone(H/H\cap C)$ is perfect. But $H/H\cap C\cong HC/C$, and so $HC/C$ cannot be isomorphic to
$\Sym_6$ or to $\PGL_2(9)$, since we know that neither of these has a perfect commuting graph (by Table~\ref{t: quasisimple}). We therefore see that $HC/C\cong \Mat_{10}$.

It is now also clear that $G/C$ cannot be isomorphic to $\mathrm{P}\Gamma\mathrm{L}_2(9)$, since otherwise there would
be a subgroup $H<G$ such that $H/C\cong \Sym_6$. So we have shown that $G/C$ is congruent either to $\Alt_6$ or
to $\Mat_{10}$, as claimed.
\end{proof}




%


\newcommand{\etalchar}[1]{$^{#1}$}

\end{document}